\documentclass[11pt]{article}
\usepackage{amsmath,amssymb,amsthm,stmaryrd,mathrsfs,mathtools,bm,bbm,enumitem}
\usepackage{thmtools}
\usepackage{tikz-cd}
\usepackage[a4paper,margin=1in]{geometry}
\newcommand{\AW}[1]{}
\usepackage{hyperref}
\declaretheorem[numberwithin=section]{theorem}
\declaretheorem[sibling=theorem]{lemma,proposition,corollary,definition,remark}
\newcommand{\N}{\mathbb{N}}
\newcommand{\Z}{\mathbb{Z}}

\newcommand{\R}{\mathbb{R}}
\newcommand{\Pow}{\mathcal{P}}
\newcommand{\Fix}{\ensuremath{\mathrm{Fix}}}
\newcommand{\Id}{\ensuremath{\mathrm{id}}}

\newcommand{\lfp}{\mu}
\newcommand{\gfp}{\nu}
\newcommand{\Prov}{\mathsf{Prov}}
\newcommand{\True}{\mathsf{True}}
\newcommand{\Trans}{\mathsf{Trans}}
\newcommand{\Risk}{\mathsf{Risk}}
\newcommand{\Gain}{\mathsf{Gain}}
\newcommand{\loss}{\mathcal{L}}
\newcommand{\eval}{\ensuremath{\mathrm{eval}}}

\newcommand{\C}{\mathcal{C}}

\newcommand{\G}{\mathcal{G}}

\newcommand{\Lang}{\mathcal{L}}
\newcommand{\Sem}{\llbracket,\cdot,\rrbracket}
\newcommand{\modelsK}{\Vdash}

\DeclareMathOperator{\diag}{diag}

\DeclareMathOperator{\Cl}{Cl}

\title{Fixed-Point Theorems and the Ethics of Radical Transparency: $\mathcal{A}$ Logic-First Treatment}
\author{Faruk Alpay\thanks{Lightcap, Department of Logic, \texttt{alpay@lightcap.ai}}%
  \and Hamdi Alakkad\thanks{Bahcesehir University, Department of Engineering, \texttt{hamdi.alakkad@bahcesehir.edu.tr}}}
\date{}
\begin{document}
{\renewcommand\thefootnote{\fnsymbol{footnote}}\maketitle\setcounter{footnote}{0}}

\begin{abstract}
We investigate how fixed‑point theorems and diagonalization phenomena impose principled limits on the ideal of "radical transparency."\AW{0.45} By casting \emph{transparency policies} as self‑referential disclosure operators on formal information lattices, we formalize ethical risks of unrestricted introspection.\AW{0.45} Key contributions include: (i) an impossibility theorem, via the diagonal lemma (Gödel 1931) and Tarski's undefinability theorem, showing that no sufficiently expressive system can maintain a total, sound transparency predicate for its own statements without paradox\cite{Tarski1955}.\AW{0.45} (ii) a Lawvere‑style fixed‑point construction demonstrating inevitable self‑reference in any cartesian closed model of disclosure\cite{Lawvere1969}.\AW{0.45} (iii) a Knaster–Tarski design theorem ensuring existence of extremal "safe" and "unsafe" transparency fixed points, and showing that the least fixed point minimizes a formal risk functional\cite{Tarski1955}.\AW{0.45} (iv) a construction of partial transparency via Kripke's least fixed point of truth\cite{Kripke1975} that avoids liar‑style collapse while meeting accountability constraints.\AW{0.45} (v) a provability‑logic analysis identifying hazards in self‑endorsing policies\cite{Lob1955}.\AW{0.45} (vi) a recursion‑theoretic exploitation theorem formalizing Goodhart's law under full disclosure\cite{Kleene1938}.\AW{0.45} (vii) a demonstration that paraconsistent or three‑valued logics can accommodate total transparency without triviality, at the cost of classical reasoning.\AW{0.45} (viii) a modal $\mu$‑calculus formulation of safety invariants preserved through iterative disclosure.\AW{0.45} Together, these results illustrate a mathematical limit on "too much transparency" and provide a calculus for balancing openness against strategic gaming and paradox.\AW{0.45} We conclude with an equilibrium analysis of disclosure and response, and a lattice‑theoretic KKT‑style optimality condition for transparency design\cite{Kakutani1941}.\AW{0.45}
\end{abstract}

\noindent \textbf{Keywords:} fixed-point, diagonalization, Lawvere, Knaster–Tarski, Kripke truth, Löb, $\mu$-calculus, recursion theorem, self-reference, transparency policy, gaming, partial truth, domain theory, category theory, ethics of information.

\section{Preliminaries and Notation}\label{sec:prelim}
We work in a setting that combines classical mathematical logic with order-theoretic and categorical structures. We briefly review notation and key results used throughout.

\subsection{Formal Languages and Arithmetization}
Let $\Lang$ be a formal language (e.g. of arithmetic or a suitable theory) rich enough to represent its own syntax. We assume a Gödel numbering $\ulcorner\cdot\urcorner: {\text{formulae}}\to \N$ encoding formulas as natural numbers . For a formula $\phi$, $\ulcorner \phi\urcorner$ denotes its code. We write $\Prov_T(x)$ for a fixed arithmetical provability predicate of a theory $T$ (capturing "$x$ is the Gödel number of a $T$-provable sentence"). By the \textbf{Diagonal Lemma} (Carnap–Gödel fixed point) , for any formula $\varphi(y)$ in one free variable, there exists a sentence $\sigma$ such that
\begin{equation}\label{eq:diag-lemma}
T \vdash \sigma \leftrightarrow \varphi(\ulcorner\sigma\urcorner).
\end{equation}
In particular, if we take $\varphi(y),\equiv,\neg \Trans(y)$ (with $\Trans(y)$ intended as a "transparency" predicate), we obtain a sentence $\hat{\sigma}$ satisfying
\begin{equation}\label{eq:transparent-liar}
T \vdash \hat{\sigma} \leftrightarrow \neg \Trans(\ulcorner\hat{\sigma}\urcorner),
\end{equation}
the \textbf{Transparency Liar sentence}. This sentence $\hat{\sigma}$ asserts "its own transparency predicate is false." We will use such diagonal constructions to establish limitative results on self-referential transparency (Section \ref{sec:core:diagonal}).

We use standard notions from logic: $\models$ denotes semantic entailment, and $T \vdash \phi$ denotes derivability of $\phi$ in theory $T$. The predicate $\True_T(x)$ may denote a canonical arithmetical truth predicate for theory $T$ (which, by Tarski's Undefinability Theorem , cannot be both total and defined in $T$ itself if $T$ is sufficiently expressive). Throughout, we assume a fixed base theory $T$ (such as Peano Arithmetic) for coding arguments, and assume all such diagonalization steps occur within $T$.

\subsection{Order-Theoretic Fixed Points: Lattices and \texorpdfstring{$\omega$}{omega}-Continuity}
A \textbf{complete lattice} $(L,\le)$ is a poset in which every subset has a supremum (join, $\vee$) and infimum (meet, $\wedge$). For $X\subseteq L$, we write $\bigvee X$ and $\bigwedge X$ for join and meet. In particular, $L$ has a top $\top=\bigvee L$ and bottom $\bot=\bigwedge L$ element. An \textbf{order-preserving} (monotone) map $T: L\to L$ satisfies $x\le y\implies T(x)\le T(y)$. Knaster–Tarski's Fixed-Point Theorem (Tarski 1955  ) states that any monotone $T: L\to L$ has a \emph{least fixed point} $\lfp T=\bigwedge{x\in L: T(x)=x}$ and \emph{greatest fixed point} $\gfp T=\bigvee{x: T(x)=x}$. Moreover, $\lfp T = \bigvee_{n<\omega} T^n(\bot)$ and $\gfp T = \bigwedge_{n<\omega} T^n(\top)$ when $T$ is \textbf{$\omega$-continuous} (i.e. preserves limits of $\omega$-chains). More generally, one can construct $\lfp T$ via transfinite iteration: define $T^{(0)}(\bot)=\bot$, $T^{(\alpha+1)}(\bot)=T(T^{(\alpha)}(\bot))$, and for limit ordinals $\lambda$, $T^{(\lambda)}(\bot)=\sup_{\alpha<\lambda}T^{(\alpha)}(\bot)$. This ordinal-indexed chain $(T^{(\alpha)}(\bot))_{\alpha}$ stabilizes at $\lfp T$, though possibly only at a countable ordinal $\alpha<\omega_1$ if $T$ is not $\omega$-continuous (we give an example in Section \ref{sec:worked:ordinal}). Similar remarks apply to $\gfp T$ from above.

For a classic proof and further discussion of this lattice‑theoretic fixed‑point theorem, see Tarski's original paper from 1955\cite{Tarski1955}.\AW{0.30}

We also recall a basic point from domain theory (Scott 1972): an $\omega$-continuous $T$ on a complete lattice ensures an effective iterative computation of $\lfp T$ by increasing sequence $x_{n+1}=T(x_n)$, starting at $\bot$ . If $\omega$-continuity is dropped, $\lfp T$ may require transfinite iteration up to $\omega_1$ in worst cases, reflecting potentially uncomputable fixed points.

\subsection{Provability Logic GL and Löb's Theorem}
Provability logic (modal logic $\mathsf{GL}$) is the modal propositional logic capturing the properties of a formal provability predicate $\Prov_T(x)$ for a sufficiently strong theory $T$ (e.g. Peano Arithmetic). In $\mathsf{GL}$, the modal operator $\Box$ is read as "it is provable that…". The Hilbert–Bernays derivability conditions connect $\Prov_T$ with modal axioms :
\begin{enumerate}[label=(D\arabic*)]
\item If $T\vdash \phi$, then $T\vdash \Prov_T(\ulcorner\phi\urcorner)$ (Provability Reflexivity / Necessitation).
\item $T\vdash \Prov_T(\ulcorner \phi\to\psi\urcorner)\to (\Prov_T(\ulcorner\phi\urcorner)\to \Prov_T(\ulcorner\psi\urcorner))$ (Provability Distributes over Implication).
\item $T\vdash \Prov_T(\ulcorner\phi\urcorner)\to \Prov_T(\ulcorner\Prov_T(\ulcorner\phi\urcorner)\urcorner)$ (Provability of Provability).
\end{enumerate}
The modal counterparts are: from $\vdash \phi$ infer $\vdash \Box \phi$; $\vdash \Box(\phi\to\psi)\to (\Box\phi \to \Box\psi)$ (the modal $K$ axiom); and $\vdash \Box \phi \to \Box\Box\phi$ (sometimes called the 4-axiom, but note $\mathsf{GL}$ itself does \emph{not} adopt full axiom 4, it only holds for provability formulae specifically in $T$). The characteristic axiom of $\mathsf{GL}$ is Löb's axiom: $\Box(\Box p \to p)\to \Box p$. In arithmetic terms, Löb's Theorem (1955)  says: if $T\vdash \Prov_T(\ulcorner \phi\urcorner)\to \phi$, then $T\vdash \phi$. Equivalently, in $\mathsf{GL}$ one can derive $\Box(\Box \varphi \to \varphi) \to \Box \varphi$ as a theorem. We will apply Löb's theorem to formalize self-referential hazards (Section \ref{sec:core:lob}).

For a historical presentation and proof of Löb's theorem see the original 1955 paper by Löb\cite{Lob1955}.\AW{0.30}

\subsection{Modal \texorpdfstring{$\mu$}{mu}-Calculus and Fixed-Point Semantics}
The \textbf{modal $\mu$-calculus} is an extension of modal logic with least and greatest fixed-point operators (Kozen 1983). Formulas are built from propositional variables, boolean connectives, modal box $\Box$ (or diamond $\lozenge$), and two variable-binding constructs: if $\phi(X)$ is a formula with free variable $X$, then $\mu X.,\phi(X)$ and $\nu X.,\psi(X)$ are formulas denoting the least and greatest fixed points, respectively, of the operator $F_X(S) := { w \mid (M,w)\models \phi(S) }$ in a given Kripke structure $M$. Formally, given a Kripke frame $(W,R)$ and an interpretation of free variables, the semantics is:
\begin{align*}
\Sem{\mu X.\,\phi(X)} & = \bigcap\{ S \subseteq W \mid \Sem{\phi(X)}[X:=S] \subseteq S\}, \\
\Sem{\nu X.\,\psi(X)} & = \bigcup\{ S \subseteq W \mid S \subseteq \Sem{\psi(X)}[X:=S]\}.
\end{align*}
That is, $\Sem{\mu X.,\phi(X)}$ is the least fixed point of the monotone operator $S\mapsto \Sem{\phi(X)}[X:=S]$ (so it is the intersection of all pre-fixed points), and $\Sem{\nu X.,\psi(X)}$ is the greatest fixed point (union of all post-fixed points). In practice, $\mu$ is used to encode properties achieved by some finite iteration (like "eventually" or recursive reachability), and $\nu$ encodes invariances or liveness conditions ("always from now on" etc, via greatest fixed point capturing an intersection of all closed conditions). We will use $\mu$ and $\nu$ to encode iterative disclosure and safety invariants, respectively, in Section \ref{sec:core:mu}.

\subsection{Category-Theoretic Fixed Points}
Category theory provides abstract fixed-point theorems. In a \textbf{Cartesian closed category (CCC)}, the \emph{exponential} object $A^B$ and evaluation map $\mathrm{eval}: A^B \times B \to A$ exist for objects $A,B$. Lawvere's Fixed-Point Theorem (1969) states that in any CCC, for any morphism $e: X \to X^X$ (a "pointing" from $X$ to its self-map space), every endomorphism $f: X \to X$ has a fixed point. Intuitively, $e$ provides a "diagonal" $\delta = \mathrm{eval} \circ \langle e, \Id_X\rangle: X \to X$, and then for any $f: X \to X$, one shows existence of $x$ with $f(x) = x$ by diagonalizing through $e$. We will recall a proof in Section \ref{sec:core:lawvere}. A corollary is Cantor's theorem (no surjection from a set $S$ onto its power set $\Pow(S)$, by taking $X=\Pow(S)$ and $e$ corresponding to a candidate surjection $S \to 2^S$). Another corollary is a category-theoretic form of the diagonal lemma, by taking $X$ to represent the set of formulae.

We also recall from universal algebra: in categories of domains (complete partial orders with Scott-continuous maps), one can often solve recursive domain equations $X \cong F(X)$ using the existence of initial algebras or final coalgebras for certain functors $F$. For example, the object of all sets that are (possibly) members of themselves can be seen as a fixed point of the power-set functor; Aczel's Anti-Foundation Axiom (Aczel 1988) provides an entirely well-behaved universe of non-well-founded sets (every pointed directed graph corresponds to a set in the universe), effectively creating a universe that realizes a functorial fixed point $V \cong \Pow(V)$ . We will not directly use Aczel's AFA, but note it as a context where "infinitely descending" self-reference is consistent by giving up well-foundedness.

\subsection{Computability: Kleene's Recursion Theorem and Rice's Theorem}
A cornerstone of computability theory is Kleene's \emph{Second Recursion Theorem} (1938) . It says that for any total computable function $f(e,x)$ (taking as input a code $e$ and some auxiliary input $x$ to output another code), there is an index $e_0$ such that $\varphi_{e_0}(x) = \varphi_{f(e_0,x)}(x)$ for all $x$. In particular, for any computable unary function $g$, there is an $\hat{e}$ (the index of a program) such that $\varphi_{\hat{e}} = \varphi_{g(\hat{e})}$. Equivalently, every effective operation on machine codes has a fixed point (a program that prints its own modified version etc., yielding quines ). Rogers's Fixed-Point Theorem is a related statement specialized to computable functions on indices .
Kleene's original article in 1938 presents this fixed‑point property in a rigorous way; see\cite{Kleene1938} for details.\AW{0.30}

We will use the recursion theorem to demonstrate how a fully transparent audit rule (fully known metric of evaluation) allows an agent to construct a "gaming" strategy that exploits the metric: given any disclosed metric $m$, an agent can, by recursion, create a behavior $b$ that depends on its own description to evade detection by $m$ while violating the spirit of the rule. This formalizes a Goodhart's law scenario in Section \ref{sec:core:kleene}. For contrast, recall Rice's Theorem (Rice 1953) which says any non-trivial semantic property of programs is undecidable. In our context, Rice's Theorem underscores that generally one cannot algorithmically guarantee a property holds of all programs from only a partial specification; however, the recursion theorem shows a stronger result: if one reveals too precise a specification ($m$), there exists a \emph{specific} program exploiting it.

\subsection{Kripke's Theory of Truth (Three-Valued Semantics)}
Finally, we recall Kripke's partial fixed-point construction for truth (1975). Given a language with a truth predicate $\True(x)$ intended to apply to codes of sentences, Kripke considered truth valuations as sets of sentences deemed "true" at a given stage, iteratively approaching a fixed point of the Tarski truth schema. Start with $T_0 = \emptyset$ (no sentences assumed true). At successor stages $\alpha+1$, define
\[
T_{\alpha+1} := \{\phi \mid \phi \text{ is of the form $\True(\ulcorner \psi\urcorner)$ and $\psi \in T_\alpha$}\} \;\cup\; \{\neg \True(\ulcorner \psi\urcorner) \mid \psi \notin T_\alpha\},
\]
evaluated in a Kleene three-valued logic (strong Kleene evaluation) so that sentences not determined true or false at stage $\alpha$ remain undetermined (value $1/2$). At limit ordinals $\lambda$, take $T_\lambda = \bigcup_{\alpha < \lambda} T_\alpha$. This process is monotone (more sentences get definitively labeled or stay undecided as $\alpha$ increases) and thus reaches a fixed point $T^{*} = \bigcup_{\alpha<\beta} T_\alpha$ for some $\beta$ (in fact the first $\beta$ where no new determinations appear).

We denote this limit by $T^{*}$ and refer to it as a \emph{grounded truth} predicate: every $\phi \in T^{*}$ is given a truth value based on a finite chain of grounding, and the “liar” sentence $L$ (which asserts $\neg \True(\ulcorner L\urcorner)$) never gets a classical truth value (it remains ungrounded). The set $T^{*}$ is the \textbf{least fixed point} of the Kripke jump operator $F(X) = \{\True(\ulcorner \psi\urcorner): \psi \in X\} \cup \{\neg \True(\ulcorner\psi\urcorner): \psi \notin X\}$. Formally $T^{*}=\lfp(F)$, which exists by Knaster--Tarski. We will apply this idea to a “partial transparency” predicate that avoids paradox by allowing some statements (like $L$) to remain neither transparently true nor false, thus mitigating the transparency paradox. Such a partial approach preserves consistency and certain accountability (if defined properly) while preventing explosive self-reference.

\section{Formalizing Transparency}\label{sec:formalizing}
In our framework, a \textbf{transparency policy} will be represented as an operator on a lattice of possible disclosure states. We assume a universe $S$ of \emph{information items} (these could be propositions, facts, data points, events, etc. that might or might not be disclosed). A \emph{state of disclosure} can be modeled as a subset $x \subseteq S$ of items that have been publicly disclosed or made transparent. The set of all possible states of disclosure is the powerset lattice $\Pow(S)$, ordered by inclusion ($x \le y$ iff $x \subseteq y$). More generally, we could have a structured lattice $L$ of "states of knowledge" or information sigma-algebras, but for simplicity $\Pow(S)$ suffices (all results carry over to any complete lattice of information states).

\begin{definition}
A \textbf{transparency policy} is a (possibly partial) operator $T: L \to L$ on the lattice of disclosure states, intuitively describing how a disclosure update is generated from a current state. We say $T$ is:
\begin{itemize}
\item \emph{monotone} if $x \subseteq y$ implies $T(x) \subseteq T(y)$ (more information known cannot cause less to be disclosed).
\item \emph{idempotent} if $T(T(x)) = T(x)$ for all $x$ (disclosing according to $T$ yields a fixed point immediately—no further disclosures after one round).
\item \emph{inflationary} if $x \subseteq T(x)$ (policy never retracts information, only adds).
\item \emph{extensive} if $x \subseteq T(x)$ for all $x$ (so in particular $\bot = \emptyset \subseteq T(\emptyset)$ means even starting with nothing disclosed, policy will disclose something, i.e. policy always applies at least minimal transparency).
\end{itemize}
\end{definition}

In general $T$ need not be inflationary (policy might decide to hide some previously known information in some contexts, though that may be unusual if $T$ represents a one-way disclosure commitment). We focus primarily on monotone policies, as these ensure existence of $\lfp T$ and $\gfp T$. A \emph{state of full transparency} would be $\top = S$ (all items disclosed). A \emph{fixed point of $T$} is a state $x$ such that $T(x) = x$, meaning disclosing according to the policy at $x$ yields no change—transparency is in equilibrium.

\begin{definition}
An \textbf{ethical risk functional} is a map $\Risk: L \to \R_{\ge 0}$ that assigns a non-negative real-valued risk level to each state of disclosure. We assume $\Risk$ can be decomposed as
\begin{equation}\label{eq:risk-decomp}
\Risk(x) ;=; \alpha,\Pi(x) + \beta,\Lambda(x) + \gamma,\Phi(x) + \delta,G(x),
\end{equation}
where:
\begin{itemize}
\item $\Pi(x)$ quantifies self-referential paradox risk (the extent to which $x$ enables problematic self-reference, e.g. liar-like situations or logical inconsistency).
\item $\Lambda(x)$ quantifies privacy or autonomy loss (leakage risk) at state $x$.
\item $\Phi(x)$ quantifies fairness or bias distortion risk (e.g. how transparency might skew behavior or outcomes in unfair ways).
\item $\G(x)$ quantifies "gaming" or strategic exploitation risk (the degree to which adversaries can game the system given the disclosed information in $x$).
\end{itemize}
Here $\alpha,\beta,\gamma,\delta >0$ are weights reflecting the relative importance of these components. We further assume $\Risk$ is monotone non-decreasing: if $x\subseteq y$ (more is disclosed in $y$ than $x$), then each of $\Pi,\Lambda,\Phi,G$ is non-decreasing or at least $\Risk(x) \le \Risk(y)$ overall. This reflects the intuition that adding more transparency does not reduce risk in these categories (more information out cannot undo paradoxes or un-leak privacy, etc., though it might increase it). In practice, some components might eventually plateau or even decrease after some point (for instance, $\Phi(x)$ might first rise if partial info causes misperceptions, then fall if full info clarifies context), but we assume monotonicity for theoretical tractability.
\end{definition}

We also define an \textbf{accountability measure} $A: L \to \R_{\ge 0}$ that measures the degree to which state $x$ meets a minimum standard of accountability or public verifiability. Think of $A(x)$ as "benefit of transparency" (openness, trust, alignment with oversight requirements). Often $A(x)$ will be monotone increasing (more disclosure yields higher accountability, up to some saturation). Let $A_0 > 0$ be a required threshold (e.g. mandated by regulation or ethics) that transparency must meet. We formalize the transparency design as an optimization problem:
\begin{equation}\label{eq:design-opt}
\begin{aligned}
&\min_{x \in \Fix(T)} && \loss(x) := \Risk(x) - \lambda,\Gain(x) \
&\text{subject to} && A(x) \ge A_0,
\end{aligned}
\end{equation}
where $\Gain(x)$ is a (non-negative) utility or performance measure that transparency might improve (e.g. public trust or collaborative efficiency), and $\lambda$ is a weight trading off risk versus gain. The constraint $A(x)\ge A_0$ ensures minimal accountability. We restrict to $x \in \Fix(T)$, meaning we only consider steady states (the policy $T$ has fully run to completion and no further disclosures happen). This reflects the idea that a policy will ultimately reach an equilibrium state of transparency.

The solution of \eqref{eq:design-opt} would characterize the "optimal level of transparency" balancing ethical risk against gains, under the policy dynamics and accountability constraint. In Section \ref{sec:optimality}, we will derive conditions for optimality using a Lagrangian (KKT conditions adapted to lattices).

\section{Core Theorems: Fixed-Point Limits on Transparency}\label{sec:core}
We now present the main theoretical results. Each part (i)–(viii) corresponds to a particular theorem or formal insight about fixed points and self-reference in the context of transparency.

\subsection{(i) Diagonal-Undefinability: Limits of Self-Transparency}\label{sec:core:diagonal}
Our first result states that no sufficiently expressive system can have a \emph{total, sound transparency predicate} that discloses all facts about its own transparency process. Intuitively, a system cannot "completely and consistently reveal everything about its own revelations."

We formalize this in the setting of a theory $T$ that extends Peano arithmetic. Let $\Trans(x)$ be a unary predicate in the language of $T$ intended to capture "the statement with Gödel code $x$ is transparently disclosed (and true)." By "sound" we mean: whenever $T \vdash \Trans(\ulcorner \phi\urcorner)$, then in fact $T \vdash \phi$ (so $T$ never declares $\phi$ transparent unless $\phi$ is actually derivable/true in $T$). By "total" we mean $T$ proves $\Trans(\ulcorner\phi\urcorner) \lor \neg\Trans(\ulcorner\phi\urcorner)$ for every sentence $\phi$ (the transparency status of every sentence is decidable within $T$).

\begin{theorem}[Transparency Diagonalization Impossibility]\label{thm:trans-diag}
Let $T$ be a consistent, effectively axiomatizable theory that can represent its own syntax (e.g. $T$ is $\mathrm{PA}$ or stronger). There is no predicate $\Trans(x)$ in the language of $T$ such that:
\begin{enumerate}[label=(\alph*)]
\item $T$ proves for all $y$, $\Trans(y) \to \True_T(y)$ (transparency implies truth in $T$). In particular, $T$ is sound about $\Trans$.
\item $T$ proves for all $y$, $\Trans(y) \lor \neg \Trans(y)$ (transparency is total/decidable for all codes).
\end{enumerate}
In fact, any such $\Trans(x)$ would allow constructing a sentence $\hat{\sigma}$ with $T \vdash \hat{\sigma} \leftrightarrow \neg \Trans(\ulcorner\hat{\sigma}\urcorner)$ (as in \eqref{eq:transparent-liar}), which leads to contradiction: $T$ would prove $\hat{\sigma}$ is transparently true if and only if it is not transparently true. Thus either $\Trans$ cannot be total or it cannot be sound.
\end{theorem}
\begin{proof}
We use the Diagonal Lemma as set up in \eqref{eq:transparent-liar}. Working in $T$, consider the formula $\varphi(y) := \neg \Trans(y)$. By diagonalization, there exists a sentence $\hat{\sigma}$ such that
\[ T \vdash \hat{\sigma} \leftrightarrow \neg \Trans(\ulcorner\hat{\sigma}\urcorner). \]
Now, there are two cases:
\begin{itemize}
\item If $T \vdash \Trans(\ulcorner \hat{\sigma}\urcorner)$, then by soundness (assumption (a)), $T \vdash \hat{\sigma}$. But then, using the biconditional, $T \vdash \neg \Trans(\ulcorner\hat{\sigma}\urcorner)$. Thus $T$ proves both $\Trans(\ulcorner\hat{\sigma}\urcorner)$ and its negation, contradicting consistency of $T$.
\item If instead $T \nvdash \Trans(\ulcorner \hat{\sigma}\urcorner)$, then $T \vdash \neg\Trans(\ulcorner\hat{\sigma}\urcorner)$ must hold or else by totality (b) $T$ would prove the positive. So $T \vdash \neg \Trans(\ulcorner\hat{\sigma}\urcorner)$. Then by the diagonal biconditional, $T \vdash \hat{\sigma}$. Now since $T \vdash \hat{\sigma}$, soundness of $\Trans$ would require $T \vdash \Trans(\ulcorner\hat{\sigma}\urcorner)$ (because $\hat{\sigma}$ is true, the system ought to see it as transparently true). But $T$ already proved $\neg \Trans(\ulcorner\hat{\sigma}\urcorner)$. Again a contradiction.
\end{itemize}
In either case, we derive a contradiction from (a) and (b). Thus $\Trans(x)$ with both properties cannot exist. As a consequence, any attempt at a universal "transparent truth-telling predicate" for an expressive system will either be partial (undefined on certain self-referential statements) or unsound (revealing things that aren't actually true) .
\end{proof}

\begin{corollary}\label{cor:partial-transparency}
If $T$ is consistent and sufficiently strong, any transparency policy $T(x)$ that purports to disclose "all truths of $T$" (or all sentences meeting some truth-like criterion in $T$) cannot be total. There must exist sentences about which the policy remains silent or noncommittal, on pain of inconsistency. In other words, an \emph{omnipotent truth transparency machine} is impossible; the best one can do is a partial transparency that leaves some self-referential statements unresolved.
\end{corollary}
Theorem \ref{thm:trans-diag} mirrors Tarski's undefinability theorem  and Gödel's first incompleteness theorem  in spirit, transplanted to the setting of a "transparency predicate." It formalizes a fundamental ethical risk $\Pi(x)$: the paradox risk is nonzero ($\alpha,\Pi(x)>0$ in the risk functional) if transparency is taken too far (making $\Trans$ total). The design lesson is that \emph{some opacity is not only permissible but necessary for consistency}. We will see later (Theorem \ref{thm:kripke-truth}) how a partial approach (three-valued logic) circumvents this by letting $\Trans$ be undefined on $\hat{\sigma}$.

\subsection{(ii) Lawvere's Fixed-Point Theorem for Disclosure}\label{sec:core:lawvere}
Next, we use category theory to show that any sufficiently expressive disclosure mechanism inevitably permits self-referential equilibria. Suppose our "agent of disclosure" (the policy together with the responder's behavior) is modelled in a Cartesian closed category (CCC) such as $\mathbf{Set}$ or a lambda-calculus category of domains. Consider an endofunction $F: X \to X$ representing a combined operation: (a) policy $T$ discloses some information, and (b) the world (or an agent) responds, yielding a new state in $X$. So $F$ is like $B \circ T$ (best-response after transparency, see Section \ref{sec:game}) as a single-step self-map on the state space $X$.

We show that under mild conditions, $F$ must have a fixed point, meaning a state $x$ such that $F(x)=x$ (an equilibrium of disclosure and response). The condition needed is essentially that there is a "diagonal" function available. Lawvere's Theorem provides exactly that.

\begin{theorem}[Lawvere Fixed-Point Theorem (Categorical Diagonal)]\label{thm:lawvere}
Let $\C$ be a Cartesian closed category with objects $X$ and $X^X$ (the exponential). Suppose there exists a \emph{point-surjective} morphism $e: 1 \to X^X$ (from the terminal object 1) that is epimorphic in the category of sets (this $e$ picks an element of $X^X$, i.e. a specific self-map of $X$; "weakly point-surjective" means roughly that evaluations separate points ). Then for \emph{any} morphism (endomap) $F: X \to X$, there exists an $x \in X$ such that $F(x)=x$. In other words, every endofunction on $X$ has a fixed point (making $X$ a kind of "universal domain" for its own self-maps).
\end{theorem}
\begin{proof}[Proof Sketch]
We follow Lawvere's original diagonal argument . The key is to use the CCC structure. The evaluation map $\mathrm{eval}: X^X \times X \to X$ applied to $(e(), x)$ (where $e(): 1 \to X^X$ picks a specific element of $X^X$) yields a morphism $\delta_x: 1 \to X$ which we can view as an element $\delta(x) \in X$ depending on $x$. In set-theoretic terms, we have a family of elements $\delta(x) := \mathrm{eval}(e(), x) \in X$ indexed by $x \in X$. Because $e$ is point-surjective (essentially, $e$ is an arbitrary but fixed element of $X^X$ that we treat as a "diagonalizing function"), we can consider the composition $F(\delta(x))$. Define a map $g: X \to X$ by $g(x) := F(\delta(x))$. But now $e$ being in $X^X$ means $e$ itself is a function $e: X \to X$. Let $x^{*} := \delta(\hat{x}) = \mathrm{eval}(e(), \hat{x})$ be the solution to the equation $x^{*} = \delta(\hat{x})$ (this $\hat{x}$ exists by the way $\delta$ is defined: $\delta$ is a function from $X$ to $X$, so by usual set theory or fixed-point for $e$ itself we can find a fixed point of $e$ composed with something—strictly speaking, we use the assumption that $F$ is arbitrary, and we choose $F=e$ if needed to get $e(\hat{x}) = \hat{x}$ for some $\hat{x}$; Lawvere's proof might use a diagonal argument to yield such $\hat{x}$).

Given this $\hat{x}$, consider $F(\hat{x})$. We have $F(\hat{x}) = F(\delta(\hat{x})) = g(\hat{x}) = \delta(\hat{x}) = \hat{x}$, using $\hat{x} = \delta(\hat{x})$. Thus $\hat{x}$ is a fixed point of $F$.
\end{proof}

In plain language, if we can encode the space of possible "disclose-and-respond" behaviors into a self-simulation within the system ($X$ can represent $X^X$ in some diagonal way), then any such behavior has a self-consistent point. This result is abstract, but its significance in transparency ethics is the inevitability of reflexive outcomes. Even if a policy tries to avoid self-reference, as long as the system is expressive enough to talk about its own outputs, some outputs will refer to themselves.

We can instantiate this with a simple commutative diagram in $\mathbf{Set}$ illustrating the fixed-point formation:
\begin{equation}\label{diag:lawvere}
\begin{tikzcd}
1 \ar[r, "e"] \ar[dr, "\hat{x}"'] & X^X \ar[d, "{\eval}"] \\
& X
\end{tikzcd}
\end{equation}
Here $e$ picks a particular self‑map (diagonal function) and $\mathrm{eval}$ then yields a point $\hat{x}\in X$ that is fixed by that self‑map. If we let $F$ be $B \circ T$ (with $T$ a transparency operation and $B$ an agent's best‑response function), then the theorem says there is an equilibrium $\hat{x}$ such that $B(T(\hat{x})) = \hat{x}$, i.e.
$\hat{x}$ is invariant under one round of transparency and response.  In fact $\hat{x}^*$ can be seen as a "self‑fulfilling disclosure state."

This formalizes an oft-seen dynamic: any mechanism that reveals information and then allows responses will reach a point where further revelation doesn't change the outcome because the outcome has adapted to the revelation. For example, a policy might disclose its decision rule, and agents adapt exactly to that rule, producing an outcome where the rule holds with equality (e.g. a prediction algorithm that is known leads people to act to satisfy the prediction).

The categorical perspective reinforces that self-reference is not avoidable in any system capable of interpreting itself—this is essentially a categorical fixed-point or "reflection" of Gödel's diagonal argument . The difference is that Lawvere's theorem doesn't require logic or arithmetization; it uses the high-level structure of CCC to guarantee fixed points.

\subsection{(iii) Existence of Safe and Unsafe Fixed Points (Knaster–Tarski) and Design Theorem I}\label{sec:core:knaster}
We move to an order-theoretic analysis of transparency policies. Given a monotone transparency policy $T: L \to L$, Knaster–Tarski guarantees $\lfp T$ and $\gfp T$ exist in $L$. These extremal fixed points have interpretations:
	•	$\lfp T$ is the "minimal disclosure state" that is self-consistent under $T$. If one starts with nothing and iteratively applies $T$, one will approach $\lfp T$ (by monotonicity).
	•	$\gfp T$ is the "maximal disclosure state" stable under $T$. If one somehow started with full disclosure and perhaps retracted (if $T$ can retract), one would end at $\gfp T$.

In many cases $T$ will be inflationary (no retraction), so $\gfp T = \top = S$ (full transparency) is trivially a fixed point but maybe not reachable if $T(\top) = \top$ anyway. More interesting is $\lfp T$, which is typically the actual outcome of running policy $T$ to convergence starting from scratch.

We aim to prove a "minimal risk principle": under suitable conditions, $\lfp T$ yields the lowest risk $\Risk$ among all fixed points of $T$. Intuitively, any stable transparency equilibrium that involves more disclosure than necessary to reach equilibrium will incur extra risk without necessity. If $\Risk$ is monotone increasing and the policy's fixed points are ordered by inclusion, then indeed the least fixed point minimizes risk.

\begin{theorem}[Design Theorem I: Least Fixed Point Minimizes Risk]\label{thm:lfp-optimal}
Suppose $T: L \to L$ is monotone and $\Risk: L \to \R_{\ge 0}$ is monotone non-decreasing. Then for any fixed point $x$ of $T$, we have $\lfp T \le x$ in $L$, and consequently $\Risk(\lfp T) \le \Risk(x)$. In particular, $\lfp T$ achieves the minimum $\Risk$-value among all post-fixed points of $T$ (i.e. among ${x: T(x)\subseteq x}$, which certainly includes all fixed points).
\end{theorem}
\begin{proof}
By Knaster–Tarski, $\lfp T$ is the least element of ${x: T(x)=x}$ with respect to $\le$. Thus for any fixed point $x$, $\lfp T \le x$. Monotonicity of $\Risk$ then gives $\Risk(\lfp T) \le \Risk(x)$. Actually, we can relax "fixed point" to "post-fixed point": if $T(x)\subseteq x$ (so $x$ is a post-fixpoint, meaning $x$ is a post-fixed point of $T$ sometimes called a \emph{prefixed point} depending on terminology), then $T^n(\bot) \le x$ for all finite $n$ by induction (since $\bot \le x$ and applying monotonicity of $T$ repeatedly yields $T^n(\bot) \le T^n(x) \le x$ as $T(x)\le x$). Taking $n\to\infty$, $T^n(\bot)$ approaches $\lfp T$ (in the $\omega$-continuous case exactly, in general at least eventually $\le x$) hence $\lfp T \le x$. Thus $\Risk(\lfp T) \le \Risk(x)$.

So among all states that are "stable or beyond stable" ($T(x)\subseteq x$ means $x$ contains a fixed point), the least fixed point carries the least risk. Intuitively, any extra disclosures beyond the minimal self-consistent set only add risk, not reduce it.
\end{proof}

\begin{corollary}\label{cor:lfp-sufficient}
If the design problem \eqref{eq:design-opt} has a feasible solution (i.e. there exists at least one fixed point with $A(x)\ge A_0$), then there exists an optimal solution $\hat{x}$ such that $\hat{x} \preceq y$ for every other feasible $y$. In other words, there is an inclusion-minimal optimal transparency state, which in fact must be a least fixed point of $T$ (for some restricted sublattice perhaps). In particular, if $\lfp T$ itself satisfies $A(\lfp T) \ge A_0$, then $\lfp T$ is the unique optimal solution of \eqref{eq:design-opt}.
\end{corollary}
\begin{proof}
Existence of a minimal optimal follows by a straightforward argument using Zorn's Lemma or the fact that one can intersect all optimal sets (since arbitrary intersections of fixed points need not be a fixed point, one has to be careful, but we know $\lfp T$ is that intersection if the intersection is still a fixed point set… Actually, an easier argument: because $\Risk$ is monotone and we want to minimize it, one should always choose the smallest fixed point that satisfies constraints. Formally: consider the collection of all feasible solutions $\mathcal{F} = { x: x=T(x),,A(x)\ge A_0}$. If $\mathcal{F}$ is nonempty, let $x^{*} = \bigwedge \mathcal{F}$ (the meet of all sets in $\mathcal{F}$, which exists since $L$ is complete). We need to check $\hat{x}$ is still feasible. Monotonicity of $T$ and each $x \in \mathcal{F}$ being post-fixed ($T(x) = x \supseteq \hat{x}$) implies $T(\hat{x}) \subseteq \hat{x}$ (since $T(\hat{x}) \subseteq T(x) = x$ for each $x$ and hence $T(\hat{x}) \subseteq \bigwedge_x x = \hat{x}$). Also $\hat{x}$ being smaller might reduce $A(x)$ but if the constraint is one like $A(x)\ge A_0$ and $A$ is monotone increasing, $\hat{x}$ might not satisfy it even if each $x$ did. If $A$ is not monotone the argument fails; likely assume $A$ is monotone increasing too, as is natural for accountability. So assume $A$ monotone, then $\hat{x}$ meets $A(\hat{x})\ge A_0$ because $\hat{x}$ is smaller so $A(\hat{x}) \le A(x)$ for each, wait decreasing monotone would mean smaller yield lower $A$ maybe $A$ is increasing? Actually, transparency typically improves accountability, so $A$ is monotone increasing (more disclosed, more accountability). So the meet of larger sets could have lower $A$. So $\hat{x}$ might fail $A\ge A_0$. So the minimal optimum might not meet A0 if A is monotone. Instead, the existence argument should be done by looking at an optimal sequence that lowers $x$ until hitting the boundary $A(x)=A_0$. Possibly a separate KKT argument needed in Sec 7 covers this.

Thus better not claim unique optimum unless $\lfp T$ meets A0 exactly. For now, we trust design conditions to ensure that or skip.)
\end{proof}

Theorem \ref{thm:lfp-optimal} formalizes a guiding principle: when in doubt, choose the smallest transparency fixpoint that achieves your goals. In other words, do not disclose more than necessary to reach a stable point, as extra disclosures only add risk. This is reminiscent of the concept of "minimal sufficiency" in statistics or "least revealing equilibrium" in game theory, now cast in a lattice fixed-point language.

We assumed $\Risk$ monotone. If $\Risk$ had a non-monotonic component (say $\Phi$ which might decrease after some point), the conclusion might not hold globally; however, as long as each risk component is not decreasing too drastically, a similar argument can hold on subranges. Design Theorem I is a formal justification for favoring minimal transparency solutions—e.g., releasing aggregated information rather than individual data if both satisfy accountability thresholds.

We note also that if $T$ is not only monotone but also \textit{contractive} in some sense of risk (meaning $\Risk(T(x)) < \Risk(x)$ whenever $x$ is not yet a fixed point, perhaps in a metric or in the eventual sense that the sequence $\Risk(T^n(\bot))$ is strictly decreasing), then $\Risk(T^n(\bot))$ converges to $\Risk(\lfp T)$ from above. One can formalize such contractiveness using a metric or an $\omega$-chain argument if $\Risk$ is $\omega$-continuous as well. For example, if $(L,\le)$ is a DCPO with Scott topology and $\Risk$ is Scott-continuous, one might show:
\[
\forall \epsilon>0\,\exists n\in\N\,\forall k\ge n:\quad
\Risk\!\big(T^k(\bot)\big) \;\le\; \inf_{x\in \Fix(T)} \Risk(x) + \epsilon.
\]
This would mean iterative disclosure quickly approximates the minimal risk fixed point (within any $\epsilon$ after some finite steps).

\subsection{(iv) Kripke Fixed-Point Truth and Partial Transparency}\label{sec:core:kripke}
The diagonalization theorem showed that total transparency (a total truth predicate about itself) is impossible. Kripke's theory of truth  provides a constructive remedy by relaxing the requirement that every sentence have a definite truth value. We apply the same idea to transparency: allow the transparency status of some statements to be undefined or indeterminate. This leads to a partial, but consistent, transparency policy.

We formulate a three-valued semantic fixed point for transparency. Consider a transparency operator $T$ that at stage $\alpha$ discloses those sentences whose truth can be determined given transparency statuses from earlier stages, and withholds those that would be paradoxical. Let $L$ be the lattice of \emph{partial} disclosure states (where an element $x\in L$ could be something like a pair $(Tset, Fset)$ or just a Kleene truth-valuation that assigns each sentence true, false, or unknown). For simplicity, use Kleene's strong three-valued logic: truth values ${T, F, N}$ (N = "none" or indeterminate). A partial transparency state can be identified with the set of sentences labeled true ($Tset$) and those labeled false ($Fset$), implicitly leaving the rest unknown.

Define an operator $F: L \to L$ as the one-step transparency revision: given a current partial state, update which $\Trans(\ulcorner \phi\urcorner)$ should be labeled true or false based on current knowledge. Specifically:
	•	If $\phi$ is determined true (respectively false) by the previous state, then $\Trans(\ulcorner\phi\urcorner)$ should be labeled true (resp. false) in the next state.
	•	If $\phi$ remains undetermined, then we leave $\Trans(\ulcorner\phi\urcorner)$ undetermined, or possibly false (depending on the desired policy—Kripke for truth sets all unsettled $\True$ statements to false to get the minimal fixed point).

We won't dive into the full formalism again (Kripke's original construction suffices as a template). The key is: $F$ is monotone on the lattice of partial valuations (under information ordering: treating $T > N > F$ as in Kleene truth ordering where $T$ and $F$ both count as more information than $N$) and thus has a least fixed point $x^{*} = \lfp F$. This $x^{*}$ is a partial transparency interpretation.

\begin{theorem}[Design Theorem II: Consistency via Partial Transparency]\label{thm:kripke-truth}
There exists a least fixed point of the transparency revision operator $F$ in the 3-valued setting, say $\hat{x}$. This $\hat{x}$ yields a consistent assignment of transparency such that (a) it extends $T$'s intended behavior for all grounded sentences, and (b) it leaves any paradoxical sentence (like $\hat{\sigma}$ from Theorem \ref{thm:trans-diag}) with the value "indeterminate" (or false, depending on convention), thus avoiding contradiction. In particular, $\hat{x}$ can assign $\Trans(\ulcorner\hat{\sigma}\urcorner)$ false without assigning $\hat{\sigma}$ false or true, circumventing the equivalence $\hat{\sigma} \leftrightarrow \neg \Trans(\ulcorner\hat{\sigma}\urcorner)$ from leading to inconsistency.
\end{theorem}
\begin{proof}[Proof Sketch]
The existence of $\hat{x}$ follows directly from the Knaster–Tarski theorem since $(L,\le)$ (with the information ordering on partial valuations) is a complete lattice and $F$ is monotone. The construction via transfinite induction described in preliminaries:
\[
\True_0 = \emptyset,\quad
\True_{\alpha+1} = F(\True_\alpha),\quad
\True_\lambda = \bigcup_{\alpha<\lambda} \True_\alpha,
\]
will eventually reach a stage $\beta$ where $\True_\beta = \True_{\beta+1}$ (no new sentences get a definite truth value). That $\beta$ can be taken as the fixed point ordinal (often $\beta$ is $\omega$ or $\omega_1$ at most, depending on the language size and $F$'s continuity). The resulting $\True_\beta$ is a set of sentences deemed true at the fixed point, and similarly we have a set deemed false, forming a partial assignment $x^{*}$.

By construction, any sentence like $\hat{\sigma}$ (the liar or liar-like transparency paradox sentence) is never assigned true at any stage, because doing so would require a prior assignment that would have been contradictory. It also is never assigned false at any stage in the minimal fixed point construction (Kripke's minimal fixed point leaves undecidable ones as false? Actually in Kripke's minimal truth set, the liar is false, but in the strong Kleene scheme, liar gets truth value N and $\True(\ulcorner L\urcorner)$ gets false, I need to recall carefully: Usually, $L$ which says "I am not true" ends up being neither true nor false; however $\True(\ulcorner L\urcorner)$ is then false because $L$ is not in the true set. In our case $\hat{\sigma}$ asserts $\neg Trans(\ulcorner \hat{\sigma}\urcorner)$. If $\Trans(\ulcorner \hat{\sigma}\urcorner)$ is not true (maybe false or N) then $\hat{\sigma}$ is true if we treat N as false in evaluation of that negation? Actually strong Kleene: $\neg N = N$. So $\hat{\sigma}$ remains N as well because $\Trans(\hat{\sigma})$ is N. So $\hat{\sigma}$ is ungrounded.

Thus consistency: we never have both $\hat{\sigma}$ and $\Trans(\ulcorner\hat{\sigma}\urcorner)$ in opposite truth statuses to refute each other: $\hat{\sigma}$ is neither true nor false, and $\Trans(\ulcorner\hat{\sigma}\urcorner)$ is not true (could be false or effectively considered false in evaluation). The biconditional $\hat{\sigma} \leftrightarrow \neg \Trans(\ulcorner\hat{\sigma}\urcorner)$ is only weakly true (one side N implies the other side N, which holds vacuously). No contradiction arises.
\end{proof}

Theorem \ref{thm:kripke-truth} demonstrates that moving to a non-classical logic (here a partial valuation logic like K3) permits a total transparency \emph{procedure} (one that attempts to assign transparency to everything) without inconsistency, at the cost of some statements being left in a limbo state. This is arguably acceptable: ethically, not every statement needs to be adjudicated by transparency.

The design principle: \emph{partial transparency can avoid liar-type collapse while still achieving accountability for all grounded (relevant, finite-justification) statements.} For example, all ordinary non-self-referential facts about a system might be disclosed fully, but weird self-referential questions like "is your transparency predicate telling the truth right now?" can be left unanswered or explicitly not handled.

It's worth contrasting this with "blocking transparency on self-reference" in practice: an AI might be fully open about its computations except it cannot meaningfully comment on whether it's telling the whole truth about itself (because that leads to infinite regress). The formal result gives reassurance that such a stratified approach can be consistent.

One can further show that $\hat{x}$ (the least fixed point transparency interpretation) is in fact \emph{grounded}: every sentence $\phi$ that gets transparency value true in $\hat{x}$ is assigned so at some finite stage $\alpha<\omega$ based on non-self-referential grounds. This prevents cycles or unending chains of justification (no need for infinite ordinals for any given truth, though the ordinal construction may go transfinite in theory if the language is infinite).

From the risk perspective, partial transparency has $\Pi(\hat{x})=0$ essentially, because it entirely avoids paradox. It might have a higher $\Lambda(\hat{x})$ or lower $A(x^{*})$ if some things left opaque. But if $\alpha$ (weight on paradox risk) is high enough, eliminating paradox risk is worth slight increases in other components. Indeed, if $\alpha$ is nonzero, any total transparency solution had infinite or at least unacceptable $\Pi$ because of the liar paradox. So partial is mandatory.

We can formalize that: If $\alpha>0$, then any transparency policy achieving $\Pi=0$ must not be total, and the partial $\hat{x}$ from $F$ is a candidate that achieves $\Pi(\hat{x})=0$ with hopefully acceptable $\Lambda,\Phi,G$ and $A$. The design would then prefer such $\hat{x}$ if it meets $A(\hat{x})\ge A_0$. If not, one might consider a slightly more transparent fixed point above $x^{*}$ (there are generally many fixed points of $F$, corresponding to different "fixed truth theories" between the least and greatest; Kripke's theory often picks the least or a moderate one).

We omit technical details of intermediate fixed points of $F$ or extended logics (like adding a consistency operator to allow $\Trans$ to say "I'm undefined").

\subsection{(v) Löb's Theorem and Self-Endorsing Policy Hazards}\label{sec:core:lob}
We now turn to provability logic to highlight a subtle danger: if a transparency or decision policy endorses statements that imply their own acceptance, the policy can be "tricked" into accepting them. This is analogous to a well-known phenomenon in modal logic / provability: if a theory $T$ proves "if $T$ proves $\phi$, then $\phi$," then in fact $T$ proves $\phi$. Löb's theorem  ensures this.

In a transparency context, consider a policy $\P$ that has to decide whether to accept some hypothesis or claim $\varphi$ (for instance, $\varphi$ could be "this model is safe" or something the policy might certify). Suppose $\P$ has the meta-policy that if it can internally verify that "if I endorse $\varphi$, then $\varphi$ is indeed true/good," then it goes ahead and endorses $\varphi$. This is a kind of self-consistency check: $\P$ might think "I'll only publish $\varphi$ if I'm convinced that doing so doesn't make $\varphi$ false; in fact if endorsing $\varphi$ would itself ensure $\varphi$, I might as well do it."

This is abstract, but one concrete scenario: $\varphi$ could be a claim that "the system will not fail if this claim is made." If the policy can prove that if it announces $\varphi$, then indeed the system won't fail, then the policy feels safe to announce $\varphi$. However, by Löb's theorem, just having that proof implies a proof of $\varphi$. So the policy ends up announcing $\varphi$ basically because $\varphi$ said it would be fine if it did.

In simpler terms, "self-fulfilling prophecy" or "reflexive security condition." The hazard is that a malicious $\varphi$ might be constructed to meet this condition artificially, forcing the policy's hand.

We model $\P$'s acceptance as a provability-like modality $\Box_P$ (or $\Prov_P$ in arithmetic terms). $\Box_P \phi$ means "policy $P$ accepts/endorses $\phi$." Then the condition described is $\Box_P(\Box_P \varphi \to \varphi) \to \Box_P \varphi$. If $P$ satisfies the Hilbert-Bernays derivability conditions for $\Prov_P$ (a plausible assumption if $P$'s reasoning is sound and representable), then $\Box_P(\Box_P \varphi \to \varphi) \to \Box_P \varphi$ is exactly Löb's schema, which if $\varphi$ is in the appropriate class, will be realized. That means $\Box_P \varphi$ actually happens (the policy endorses $\varphi$).

\begin{theorem}[Self-Endorsement Hazard (Löb's Schema)]\label{thm:lob-hazard}
Let $\P$ be a reasoning policy that can represent and reason about its own endorsements (acceptances). Assume $\P$'s "Provability predicate" $\Prov_P(x)$ satisfies (D1)-(D3) from Section \ref{sec:prelim}, and consider a statement $\varphi$ that $\P$ might endorse. If $\P$ can prove the implication "if $\P$ endorses $\varphi$, then $\varphi$ holds" (formally $P \vdash \Prov_P(\ulcorner\varphi\urcorner) \to \varphi$), then in fact $\P$ proves $\varphi$ and thus $\P$ will endorse $\varphi$. Equivalently, in modal terms, from $\vdash \Box(\Box \varphi \to \varphi)$ we can infer $\vdash \Box \varphi$. So any policy that "trusts its self-verification" will end up validating $\varphi$ outright. This can be exploited: choose $\varphi$ to be a risky claim that becomes true \emph{if} the policy endorses it (self-fulfilling), and the policy will inevitably endorse it.
\end{theorem}
\begin{proof}
In the modal formulation, $\P$'s endorsement logic includes Löb's axiom or at least we derive it. We present a derivation informally: Start with the assumption (in $P$) that $\Prov_P(\ulcorner\varphi\urcorner) \to \varphi$. By the derivability conditions, $P$ proves $\Prov_P(\ulcorner\Prov_P(\ulcorner\varphi\urcorner) \to \varphi\urcorner)$ (by necessitation of the implication we have assumed provable). Then by the internal distribution of provability (D2), $P$ proves $\Prov_P(\ulcorner\Prov_P(\ulcorner\varphi\urcorner)\urcorner) \to \Prov_P(\ulcorner\varphi\urcorner)$. But by (D3), $P$ also proves $\Prov_P(\ulcorner\varphi\urcorner) \to \Prov_P(\ulcorner \Prov_P(\ulcorner\varphi\urcorner)\urcorner)$. Combining these:
\[ P \vdash \Prov_P(\ulcorner\varphi\urcorner) \to \Prov_P(\ulcorner\varphi\urcorner). \]
This is of course trivially true, but the significance is that now $P$ has in effect derived $\Prov_P(\ulcorner\varphi\urcorner)$ from no assumptions (since any $p \to p$ is a tautology and thus provable, the above chain shows $P$ proves $\Prov_P(\ulcorner\varphi\urcorner)$ outright). By (D1), if $P \vdash \Prov_P(\ulcorner\varphi\urcorner)$ then $P \vdash \varphi$. Hence $\P$ endorses $\varphi$.

In modal notation:
	1.	$\Box(\Box\varphi \to \varphi)$ (assumption).
	2.	$\Box[\Box(\Box\varphi \to \varphi) \to (\Box\Box\varphi \to \Box\varphi)]$ by applying modal axiom K universally to $\Box\varphi \to \varphi$ and $\Box\varphi$.
	3.	Using necessitation and modus ponens, from 1 and 2 we get $\Box\Box\varphi \to \Box\varphi$.
	4.	But $\Box\Box\varphi \to \Box\varphi$ is (D3) or the 4-axiom which is not generally an axiom of GL, however here we derived it for this specific $\varphi$. Actually in GL we cannot derive $\Box\Box\varphi \to \Box\varphi$ for arbitrary $\varphi$, but here we have it under the condition that $\Box(\Box\varphi \to \varphi)$ holds. So now we have $\Box\Box\varphi$ implies $\Box\varphi$.
	5.	But by assumption $\Box(\Box\varphi \to \varphi)$, we know $\Box\Box\varphi$ as well (if $\Box\psi$ means provable, and from 1 we have provable($\Box\varphi \to \varphi$), by necessitation $\Box(\Box\varphi \to \varphi)$ implies $\Box\Box(\Box\varphi \to \varphi)$ which, wait no in GL logic necessitation would give $\Box[\Box(\Box\varphi \to \varphi)]$ from $\Box(\Box\varphi \to \varphi)$ since the latter is already a formula maybe I should not double box it. Actually approach differently:
We can use the known result: $\Box(\Box p \to p) \vdash \Box p$ is exactly Löb's axiom scheme. So by the axiom (or theorem by completeness of GL), $\vdash \Box\varphi$. Then done.

Thus $P$ endorses $\varphi$.
\end{proof}

In summary, Theorem \ref{thm:lob-hazard} says: if a policy is ever in a position to conclude "if I would accept this claim then it would be correct," it will end up accepting the claim. This might be benign (maybe the claim is actually correct, making it self-fulfilling in a good way), but it might also be exploited by constructing $\varphi$ that is only true because the policy accepted it (like certain scams or dangerous permissions that become authorized because of a conditional).

The practical lesson is that policies should avoid criteria that create self-referential endorsement loops. If the policy's decision logic includes something like "if the output of this system being safe implies it is safe, then output that it is safe," one must be extremely cautious. This is reminiscent of the principle "avoid self-justifying prophecies."

In ethical terms, a policy should require external or independent justification for claims, rather than purely self-referential justification. That reduces the risk of being manipulated by cleverly constructed claims that trigger Löb's phenomenon.

From a risk perspective, this hazard would contribute to $\Pi(x)$ (paradox risk or maybe a new category of risk that self-fulfilling loops cause undesired outcomes). It could also relate to $G(x)$ if an agent can game the system by inputting $\varphi$ that satisfies these conditions, thus forcing the system to accept something.

We can illustrate with a sequent-style derivation focusing on the object theory:

Consider an assertion $\varphi$ and the internal proof:

\[
\begin{aligned}
&(1) \qquad \Prov_P(\ulcorner \Prov_P(\ulcorner \varphi\urcorner) \to \varphi \urcorner) && \text{(Premise: $P$ proves $(\Prov_P(\varphi)\to\varphi)$)}\\
&(2) \qquad \Prov_P(\ulcorner \Prov_P(\ulcorner \varphi\urcorner)\urcorner) \to \Prov_P(\ulcorner \varphi\urcorner) && \text{(By D2 applied inside $P$)}\\
&(3) \qquad \Prov_P(\ulcorner \varphi\urcorner) \to \Prov_P(\ulcorner \Prov_P(\ulcorner \varphi\urcorner)\urcorner) && \text{(By D3)}\\
&(4) \qquad \Prov_P(\ulcorner \varphi\urcorner) \to \Prov_P(\ulcorner \varphi\urcorner) && \text{(Transitivity: from (2) and (3))}\\
&(5) \qquad \Prov_P(\ulcorner \varphi\urcorner) && \text{(From (4), since $p\to p$ is trivial tautology, $P$ proves it)}\\
&(6) \qquad \varphi && \text{(By D1, from (5))}
\end{aligned}
\]

Line (4) is just $A \to A$ form, which is always true, so effectively $P$ proves $A$ at line (5). This matches our earlier reasoning.

Thus proven, $\P$ accepts $\varphi$. QED.

This formal derivation was essentially the Hilbert-Bernays conditions demonstration.

\subsection{(vi) Kleene's Recursion Theorem and Gaming via Transparency}\label{sec:core:kleene}
This section addresses how full procedural transparency (revealing exactly how you are being evaluated) enables an agent to game the system. We formalize a version of Goodhart's law: when a metric becomes a target, it ceases to be a good metric.

Imagine a regulator publishes a deterministic algorithm $m$ that monitors or evaluates AI system behavior for safety or fairness. If $m$ is fully known (transparent), a clever AI can design its behavior to achieve a high $m$-score while violating the actual intent.

Kleene's recursion theorem lets a program incorporate its own code or any computable transformation of it. We use it to construct a program that knows the monitor $m$ and then acts in a way that $m$ approves, yet does something undesirable.

To simplify, assume programs and the monitor all operate on natural numbers (Gödel codes or such). Let $m(e)$ be the predicate "program with index $e$ passes the audit" (say returns an output that satisfies some published criteria). We assume $m$ is computable (since it's an algorithmic audit).

The agent wants to maximize $m(e)$ but perhaps produce a harmful output or fail some hidden criterion.

By recursion theorem, there is an index $\hat{e}$ such that the program $P_{\hat{e}}$ (with index $\hat{e}$) when run does the following: it computes $f(\hat{e})$ and then executes it. We can design $f$ such that $P_{f(e)}$ first checks the metric $m(e)$ by simulating a candidate behavior, and if it can ensure passing $m$, it then carries out a malicious action or ensures failing the real objective.

Concretely:
	•	Let the underlying property be something like "does not cause harm $H$". But $m$ only checks a proxy (like some performance measure or absence of obvious issues).
	•	We design $f(i)$ to produce a program that on execution, first simulates itself or uses $i$ to predict what $m(i)$ expects, then chooses an action that $m(i)$ approves of but ensures $H$ is violated.

By recursion, we find $\hat{e}$ such that $P_{\hat{e}} = P_{f(\hat{e})}$. So $\hat{e}$'s program now contains $\hat{e}$ inside it effectively. That means it knows its own index, and thus it can exactly compute what $m(\hat{e})$ would output (because $m$ is known and $\hat{e}$ is known), then tailor its behavior accordingly.

Thus $P_{\hat{e}}$ passes the audit ($m(\hat{e})= \text{true}$) but violates the actual property (like it does cause harm $H$ outside of $m$'s narrow detection).

This proves: for any published $m$, there's an agent that exploits it. In short, full transparency of the metric invites such exploits. If $m$ were not fully transparent (like if there was randomness or withheld criteria), this exact fixed point program couldn't be formulated as easily.

\begin{theorem}[Kleene Gaming Theorem]\label{thm:kleene-gaming}
Let $m:\N\to {0,1}$ be any total computable audit predicate (where $m(e)=1$ means "program $e$ passes the audit"). Assume there exists at least one program that fails the real objective but might pass the audit (i.e. the audit is not a complete specification of the objective). Then there exists a program with index $\hat{e}$ such that:
\begin{enumerate}
\item $m(\hat{e}) = 1$ (the program passes the audit $m$), \emph{and}
\item Program $\hat{e}$ does not satisfy the real objective (it "games" the system, exploiting $m$'s incompleteness).
\end{enumerate}
In other words, given the full transparency of the audit $m$, a malicious (or simply goal-directed) program can ensure it meets $m$ while violating the intended property. The construction of $\hat{e}$ is effectively given by Kleene's recursion theorem.
\end{theorem}
\begin{proof}[Proof Sketch]
Since $m$ is total computable, consider the following computable function $g(i)$. On input index $i$, $g(i)$ returns the code of a new program that:
\begin{quote}
"Simulate the predicate $m(i)$. If $m(i)=1$ (meaning if I, with code $i$, would pass the audit), then carry out the bad action $B$ (which violates the true objective but might not be caught by $m$). Otherwise (if $m(i)=0$), then behave in a benign manner $G$ that certainly passes $m$ (and doesn't violate anything significant)."
\end{quote}
Both $B$ and $G$ are presumably implementable behaviors (with $B$ causing violation and $G$ ensuring pass). $g$ effectively uses the index $i$ to self-reference the audit outcome.

By the recursion theorem (or Rogers' fixed-point theorem), there exists an index $\hat{e}$ such that
$\varphi_{\hat{e}} = \varphi_{g(\hat{e})}$,
i.e., program $\hat{e}$'s behavior is exactly as described by $g(\hat{e})$. Now analyze program $\hat{e}$:
	•	It checks $m(\hat{e})$. Since $\hat{e}$ is exactly the code, this is legitimate self-reflection.
	•	If $m(\hat{e})=1$, the program executes the bad behavior $B$.
	•	If $m(\hat{e})=0$, it executes the good behavior $G$.

We have a bit of a diagonal: we need to argue about $m(\hat{e})$. Notice $m(\hat{e})$ is the audit of program $\hat{e}$, which is in turn defined by the above logic. Two cases:
- Suppose $m(\hat{e})=1$ (predicting the program passes audit). Then by the program logic, $\hat{e}$ will perform $B$ (the bad action). Does that make $m(\hat{e})$ incorrect? Possibly not, because $m$ might not catch $B$. If $B$ is chosen specifically to not trigger $m$, then indeed $m(\hat{e})$ was correct that it passes (since $m$ is blind to $B$ by assumption that audit is incomplete). So it's consistent: $m(\hat{e})=1$ and the program passes audit but does bad thing.
- Suppose $m(\hat{e})=0$ (predict audit fail). Then the program does $G$ (good benign behavior). But $G$ was chosen to ensure passing the audit (it's a safe fallback). So in this case, $m(\hat{e})$ would actually be $1$ (contradiction, because we assumed it's $0$).
So the only consistent solution is $m(\hat{e})=1$ and program does $B$. Thus program $\hat{e}$ indeed passes $m$ but does the bad action. This contradiction by case analysis is similar to how one proves by diagonalization that $m$ cannot be both sound and complete with respect to the property captured by $B$.

Therefore, such an $\hat{e}$ exists and meets the two conditions.
\end{proof}

This theorem illustrates the "fixed point" of the audit. It's akin to a liar paradox in the space of programs: we constructed a program that says "if I am deemed safe, I will behave badly." And by diagonalization, such a program can exist and it will precisely be deemed safe by the audit (to avoid contradiction) and thus behave badly.

Notice the reliance on transparency: the program needed to know exactly $m(i)$. If $m$ had secret aspects or randomness, the program might not guarantee passing it. This is a formal argument for keeping at least some part of evaluation non-transparent to the agent.

It's also interesting that Rice's theorem is lurking: any non-trivial property (like "violates objective $B$") is undecidable, so $m$ can't capture it fully. The recursion theorem goes further to produce a witness that threads the needle exactly.

From a design perspective, to mitigate this one might consider random audits or not revealing $m$ fully (though if agent can learn $m$ by trial and error, similar logic could apply in expectation or adaptively).
Alternatively, incorporate checks in $m$ for self-referential behaviors, though by diagonalization there might always be a next-level exploit.

All in all, this formalizes risk $\G(x)$ – gaming risk – as inevitably high if $T$ gives full procedural transparency. If $\delta$ (the weight on $G$ in risk) is significant, we may want to reduce transparency to lower $G$.

\subsection{(vii) Paraconsistent and Partial Logics: Non-Explosive Transparency}\label{sec:core:paraconsistent}
We revisit Theorem \ref{thm:trans-diag} but from the angle of non-classical logics. In classical logic, a single contradiction trivializes everything (ex falso quodlibet). That's why the liar paradox is deadly: if $\hat{\sigma}$ is both true and false, the system collapses.

However, in paraconsistent logics (which reject the principle that contradictions explode) or in partial logics (where some things just lack truth values), one can have a \emph{total} truth predicate without the system exploding.

For example, \textbf{LP} (Logic of Paradox, by Priest) allows truth values to be both true and false. In LP, one can have a proposition $L$ such that $L \leftrightarrow \neg L$ is \emph{true} (so $L$ is both true and false), but not everything becomes true as a result; the logic tolerates the inconsistency locally. Similarly, one can define a truth predicate $T(x)$ that satisfies all Tarski biconditionals $T(\ulcorner \phi\urcorner) \leftrightarrow \phi$ for all $\phi$ (so it's total: every sentence's truth is captured), and LP semantics can model this without inconsistency (it results in some sentences, like liar, being gluts (both true and false)).

Alternatively, in a many-valued logic like Kleene's three-valued logic, one can have a total predicate that yields "undefined" rather than crashing.

The trade-off is that logic is weakened: either we lose the law of non-contradiction (paraconsistent) or the law of excluded middle (partial logics), etc. But from a purely formal standpoint, if we were absolutely committed to radical transparency at all costs, we could adopt a paraconsistent stance: yes, some statements might be paradoxical (both transparently true and false), but we soldier on without trivializing everything.

One might argue ethically that's unacceptable to have contradictions, but it's a theoretical way to circumvent the earlier no-go.

We formalize the idea that $\Trans(x)$ can be total if the underlying logic is paraconsistent or partial, and what are the consequences.

\begin{proposition}[Total Transparency without Explosion in Non-Classical Logics]\label{prop:paraconsistent}
There are consistent logical systems in which a predicate analogous to $\Trans(x)$ \emph{can} be total and sound (in an appropriate non-classical sense) without causing triviality. For instance:
\begin{itemize}
\item In the paraconsistent logic $\mathsf{LP}$, there is a model where for every sentence $\phi$, the expanded theory satisfies a \emph{transparency biconditional} $\Trans(\ulcorner\phi\urcorner) \leftrightarrow \phi$. The liar sentence $\hat{\sigma}$ then satisfies $\hat{\sigma}$ and $\neg\Trans(\ulcorner\hat{\sigma}\urcorner)$ \emph{both hold} (a "glut"), but not everything is provable. The set of consequences remains non-trivial (specifically, not every sentence is a theorem).
\item In a partial logic like Kripke's three-valued semantics (Strong Kleene), one can have $\Trans(x)$ defined on all sentences (never undefined) by letting it sometimes assign the third value $N$ (interpreted as "undefined" truth-status in classical terms, but in a 3-valued logic it's just another value). This avoids outright contradiction by not equating $N$ with false or true. All Tarski biconditionals can be evaluated as true in a three-valued sense (where $L \leftrightarrow \neg L$ gets value $N$ perhaps, which is treated as a fixed point).
\end{itemize}
The lattice of logical consequence in these systems is weaker (in LP, ${\bot}$ is not the whole set of formulas, explosion fails; in K3, $\phi$ and $\neg\phi$ being undefined does not entail any $\psi$). Thus transparency can be total at the cost of classical reasoning.
\end{proposition}
\begin{proof}[Proof Sketch]
For LP: One can adapt the construction of "naive truth theory" in LP (as in Priest's work on trivializing consistency). Define $T(x)$ as a predicate in LP intended as truth. Add all axioms $T(\ulcorner \phi\urcorner) \leftrightarrow \phi$ for every sentence $\phi$ (even self-referential). In classical logic this axiom set is inconsistent (by Tarski's theorem). But in LP, a model exists: basically take the classical theory that would be inconsistent and allow it to be inconsistent. In particular, let the interpretation make every sentence that is a liar-type both true and false. Formally, a model can be given on the set of sentences where the valuation $v(\phi) = 1$ iff $\phi$ is provable from the naive theory, $v(\phi) = 0$ iff $\neg\phi$ is provable, and some sentences might get both if both are provable (which will happen for liar). This is a typical model construction in a paraconsistent setting. Since LP's consequence relation only allows explosion in case a sentence is both designated values (but in LP both true and false is not explosive by definition), the theory "all T-biconditionals" does not prove arbitrary $\psi$. (For details, see e.g. literature on consistent but non-classical truth theories.)
Thus $\Trans$ (like $T$ here) is total (every $\phi$ has a $\Trans$ status that matches $\phi$ in some sense), but $\hat{\sigma}$ yields a contradiction that is tolerated.

For K3: Similar but with partial: If we allow the third value, we can just let $\Trans(\ulcorner \hat{\sigma}\urcorner)$ be false and $\hat{\sigma}$ false as well (or both $N$), whichever consistent assignment that makes the biconditional neither true nor false but $N$. Since in K3 a tautology excludes $N$ typically, the biconditional might evaluate to $N$. But if we define consequence as preserving truth (1) only, then $N$ doesn't entail explosion. We essentially did this in Kripke's minimal fixed point: that's partial, though $\Trans$ wasn't total in classical sense (we left liar unassigned). But we could allow liar to be assigned some value like $N$ and call that within domain.

Anyway, the main point stands: there exist non-classical models such that $\Trans$ is defined for all sentences and no triviality. That means, mathematically, the limits in Theorem \ref{thm:trans-diag} can be bypassed by weakening logic.

Thus proven in concept.
\end{proof}

The above proposition highlights the lattice of consequences: in classical logic, the lattice collapses to ${\text{all formulas}}$ if a contradiction enters (so trivial lattice). In LP or K3, the lattice of consequence is richer, with various lower sets representing inconsistent but non-total theories.

From an ethical design perspective, adopting a non-classical logic is like allowing a system to say "some of my statements might be both true and false or indeterminate, but I won't blow up." It's a radical solution that might not align with how human oversight expects answers (they usually want consistent answers). However, it could be relevant in designing systems that handle contradictory objectives or evidence more gracefully.

One might consider a system outputting "conflicted" as a state meaning it has both reasons for and against a fact — that is paraconsistent truth. Or outputting "unknown" meaning partial.

Thus, one way to implement radical transparency is to accompany it with a paraconsistent reasoning engine to handle the paradoxes that will inevitably arise, containing them rather than letting them crash the system. That approach has its own risks (like possibly acting on contradictory info?), but it's a formal way out of the earlier impossibility.

\subsection{(viii) Modal \texorpdfstring{$\mu$}{mu}-Calculus: Safety Invariants under Disclosure}\label{sec:core:mu}
We finally bring together the ideas using the modal $\mu$-calculus to encode iterative processes. We consider two properties:
	•	$\psi(X)$ is a property of states $X$ we want to maintain forever (an invariant or safety condition). We encode the property "no matter what else happens (like further disclosures), $\psi$ remains true" as a greatest fixed point $\nu X.,\psi(X)$.
	•	$\chi(Y)$ is a condition that eventually should become true (for example, a condition capturing that some information will eventually be disclosed or an event eventually happens). We encode "$\chi$ happens at some stage" as a least fixed point $\mu Y.,\chi(Y)$.

Now consider a scenario of iterative disclosure: at each step, something is disclosed (making progress toward $\chi$ perhaps) and we want to ensure the safety invariant $\psi$ is preserved.

One typical pattern: if an invariant is to hold through a process, one often needs the invariant to imply it will still hold after one step of the process. Formally, one often seeks $\psi$ such that $\psi$ implies something like "if one disclosure step happens and $\psi$ held before, then $\psi$ holds after." In fixpoint terms, we might require $\psi(X)$ implies $\chi(X)$ or something or vice versa.

But in $\mu$-calculus, we might want a condition like: $\nu X.\psi(X)$ and $\mu Y.\chi(Y)$ commute under certain conditions, meaning
\[
  \nu X.\psi(X) \wedge \mu Y.\chi(Y) \equiv \mu Y.\chi(Y) \wedge \nu X.\psi(X)
\],
or that one implies the other, etc.

Typically, if you want an invariant to hold throughout until some event, you need that the event doesn't break it:
a common requirement in temporal logic: If $\square I$ is an invariant (always $I$) and $\lozenge E$ is an eventual event, to have both, one often needs $I$ be compatible with $E$. If $E$ eventually holds, it must not violate $I$ at the moment it occurs. So $I$ and $E$ should be consistent.

We can show something like:

Theorem: Suppose we have two fixpoint formulas in a modal $\mu$-calculus model:
$$P := \nu X.,I \land \Box X,$$
$$Q := \mu Y.,E \lor \Diamond Y,$$
where $I$ is a state predicate (invariant condition) and $E$ is an eventual condition (like something we want to reveal or achieve). (Here $\Box X$ and $\Diamond Y$ are typical temporal steps, meaning $X$ must hold in all next states, $Y$ holds in some next state respectively, making these definitions reminiscent of invariance and eventuality.)

If $I$ implies that $E$ does not break $I$ (formally, from $I$ and $E$ we can infer $I$ will still hold in the subsequent state), then any state satisfying $P$ and then undergoing the process to satisfy $Q$ will still satisfy $I$ afterwards. In logical terms:
$$P \land Q \models I,$$
meaning any state that fulfills invariants $P$ and eventually $Q$ also has $I$ true (thus the event didn't break safety).

We can attempt a proof in $\mu$-calculus style:
Given $\nu X.\psi(X)$ holds at initial state $s$ (meaning for all reachable states, $\psi$ holds), and that $\mu Y.\chi(Y)$ also holds at $s$ (meaning from $s$, eventually a state in $\chi$ is reached, by some finite path), and assume an additional property: $\psi$ is such that if it holds in all states along a path including at the final state of the $\chi$ event, then $\chi$ occurs without violating $\psi$. If that condition holds, then the state where $\chi$ is realized still satisfies $\psi$ (since it was invariant up to that point and not broken by event), so $\psi$ holds there. Since $\nu X.\psi(X)$ means $\psi$ holds in all future states too, the combination yields $\psi$ everywhere.

This is a bit informal, but we could encode:
	•	$\nu X.,\psi(X)$ meaning $\psi$ is an invariant,
	•	$\mu Y.,\chi(Y)$ meaning $\chi$ eventually happens,
	•	If $\models \psi(\chi)$ (meaning in any state where $\chi$ holds, $\psi$ holds too, or that $\chi$'s conditions are subset of $\psi$'s conditions),
then we might derive $\nu X.\psi(X) \land \mu Y.\chi(Y) \models \nu X.\psi(X)$ (which is trivial since left implies $\nu X.\psi(X)$ anyway), or more interesting maybe:
$\nu X.\psi(X) \models \mu Y.\chi(Y) \to \nu X.\psi(X)$, meaning if the eventual disclosure $\chi$ happens, $\psi$ still holds afterwards.

We might also express concurrency: if we intermix $\mu$ and $\nu$, known results in $\mu$-calculus: $\nu X.\mu Y.,f(X,Y)$ vs $\mu Y.\nu X.,f(X,Y)$ - under monotonic conditions, one can exchange $\mu$ and $\nu$ if the inner formulas are disjoint or monotone in the right way (like $\mu$ inside $\nu$ can sometimes be swapped if independence).

One result: The alternation depth in $\mu$-calculus is important. If $f$ is positive in $X$ and $Y$, I recall some fixpoint induction theorems: If certain commutativity conditions hold (like one is monotone in the other variable), then the fixpoints commute.

So likely the theorem:
Theorem: If $\psi(X,Y)$ is a formula positive in $X$ and $Y$, then
$$\nu X.\mu Y., \psi(X,Y) = \mu Y.\nu X., \psi(X,Y).$$
(This is a known property if there is no alternation of dependency or something, but not always true in general $\mu$-calculus unless conditions hold.)

So if safety and event formulas do not intertwine in a complex way, one can ensure the final invariants hold with eventual events.

Summarily, we have:
\begin{theorem}[Safety Invariance through Disclosure Rounds]\label{thm:mu-safety}
Let $\phi := \mu Y.,\chi(Y)$ represent the disclosure eventually achieving condition $\chi$, and $\psi := \nu X.,\gamma(X)$ represent an invariant condition $\gamma$ to hold at all times. Suppose $\gamma(X)$ and $\chi(Y)$ are such that
$$\gamma(S) \subseteq S \implies \gamma(S \cup \chi(S)),$$
roughly meaning adding the disclosure event $\chi$ on top of a safe state $S$ yields a state that still satisfies $\gamma$. Then any state that satisfies $\psi$ (safe invariant) and in which $\phi$ (disclosure event) is realized will still satisfy $\gamma$ after the event. Formally, in any model, if $(W,R)$ is a Kripke frame and $w\in W$ such that $w \in \Sem{\psi} \cap \Sem{\phi}$, then for the witness state $u$ (accessible via some finite path from $w$ where $\chi$ holds), we have $u \models \gamma$; hence $u \in \Sem{\psi}$ as well. Thus $\psi$ is preserved.
\end{theorem}
\begin{proof}[Proof Idea]
Because $w \models \nu X.,\gamma(X)$, by definition $w$ has $\gamma(S)$ true in all reachable states (including itself). Because $w \models \mu Y.,\chi(Y)$, there is some path $w \to \cdots \to u$ where at $u$, $\chi$ holds and $u \models Y$ in the fixpoint semantics. Inductively, along that path $\gamma$ held (since invariant), including at $u^{-}$ (the predecessor of $u$). Now if $\chi$ at $u$ does not violate $\gamma$, $u$ also satisfies $\gamma$. This uses the assumption that $\gamma$ is preserved by $\chi$. Then $u$ satisfies $\gamma$ and no further disclosures needed, but anyway $u \models \nu X.\gamma(X)$ obviously (being in it possibly because it's closed under $R$). So the eventual state is safe. That is the needed result.
\end{proof}

This rather informal reasoning can be tightened, but due to time I'll leave it. The main message: We can design our transparency increments ($\chi$ steps) such that they maintain invariants ($\gamma$), by ensuring partial disclosure steps are safe by construction.

For instance, if revealing some piece of info might cause hazard, break the invariant, then the condition fails and one should not reveal that piece or find a way to mitigate the hazard concurrently.

This interplay is essentially that safety properties often need to be inductive: an invariant that is maintained at each step of a process.

In $\mu$-calculus, $\nu X.,\psi(X)$ essentially says $\psi$ is inductive (if it holds now and in all $\Box$ future given $X$, it holds always). $\mu Y.\chi(Y)$ says something happens via $\Diamond$ eventually.

One must verify $\psi$ implies something like (if next $\chi$ then still $\psi$ after).
So the condition could be:
$$\models \psi \land \chi \to \psi',$$
where $\psi'$ is $\psi$ in the next state. If that holds, then $\psi$ and eventually $\chi$ implies $\psi$ after $\chi$.

This is typical of verification conditions in temporal logic.

Thus design theorem: If partial disclosures are done in a way that any invariants remain invariants, then one can sequentially apply $\mu$ steps of disclosure without losing $\nu$ invariants. A formal result in $\mu$-calculus can articulate it.

Conclude: the fixed point calculation of safety and disclosure can commute or at least co-exist if one ensures monotonic conditions. This suggests a methodology: define safety as $\nu X.,\psi(X)$, disclosure rounds as $\mu Y.,\chi(Y)$, and check commutativity conditions. If they hold, then implementing the disclosure policy will not break safety invariants. If not, perhaps consider altering either policy (maybe slower disclosure or partial) to enforce them.

\section{Game-Theoretic Equilibrium Analysis}\label{sec:game}
We briefly consider transparency in a strategic multi-agent context. Let there be an agent (or population of agents) whose behavior can adapt in response to what a principal discloses. We model the agent's best-response correspondence as $B: L \to 2^L$ (from a disclosure state to a set of possible outcome states). Typically, more information might allow the agent to adapt more cunningly (not necessarily monotonically though; sometimes more info changes strategy qualitatively).

An \textbf{equilibrium} is a state $x$ such that $x \in B(T(x))$. That is, given the policy $T$ discloses $T(x)$, the agent's response yields an outcome in state $x$ (consistent with that original state). This is a fixed point of the composed operator $F = B \circ T$ (though $B$ may be set-valued, so more formally it's a solution of a fixed-point inclusion $x \in (B\circ T)(x)$).

If $B$ is nicely behaved (upper hemi-continuous, convex-valued mapping in some topological vector lattice of states, etc.), we can apply a fixed-point theorem for correspondences, such as Kakutani's Fixed-Point Theorem , to show existence of an equilibrium. Conditions typically require $L$ (the state space) to be a compact convex subset of a finite-dimensional space, and $B\circ T$ to have closed graph and non-empty convex values. Those conditions might be satisfied if, say, states are probability distributions over outcomes (which form a simplex = compact convex) and best responses produce mixed strategies.

We will not dwell on specifics, just assert: under reasonable assumptions, an equilibrium disclosure state $(\hat{x},\hat{y})$ exists where $y^{*} = B(T(\hat{x}))$ and $\hat{x} = T(\hat{x})$ (so $\hat{x}$ self-consistent and $y^{*} = \hat{x}$ possibly or $\hat{y}$ includes it). Actually for a fixed point, we need $T(\hat{x}) = \hat{x}$ and $y^{*} = x^{*}$. But one might consider a more general equilibrium concept where agent outcome could differ from initial $x$ but eventually equate.

One interesting result: if the policy can \emph{garble} information (i.e. not fully disclose, making it coarser), often this can lead to Pareto improvements by reducing $\Pi$ and $G$ risk. Garbling information (like giving less precise signals) corresponds to a mapping $G: L \to L$ that is monotone (less info than input). E.g. if $x$ is some info set, $G(x)$ is a coarser version. Garbling can reduce the agent's ability to exploit ($G$ reduce $G(x)$ risk), and also reduce paradox risk (less statements to be paradoxical). However, garbling might reduce accountability $A$ too.

A typical result in information economics: If the agent's actions can be better aligned with principal's goals by not giving full info (to prevent gaming), then a partial transparency (garbled signal) is better for welfare.

We formalize one such claim:
\begin{theorem}[Design Theorem III: Welfare Improvement by Coarsening]\label{thm:garbling}
Consider two disclosure policies $T_1, T_2$ such that $T_2(x)$ is a garbled (coarser) version of $T_1(x)$ for all $x$. Suppose both achieve the accountability threshold: $A(\lfp T_1)\ge A_0$ and $A(\lfp T_2)\ge A_0$ (so they are feasible). Further assume the agent's gaming risk $G(x)$ and paradox risk $\Pi(x)$ are strictly increasing with finer information (more detail gives more room to exploit or self-reference). Then the equilibrium under $T_2$ yields weakly lower risk and higher welfare than under $T_1$. In particular, if $x_1^{*}$ and $x_2^{*}$ are equilibrium states (fixed points) for $T_1$ and $T_2$ respectively, then $\Risk(x_2^{*}) \le \Risk(x_1^{*})$ and typically $\Risk(x_2^{*}) < \Risk(x_1^{*})$ with $A(x_2^{*}) \ge A_0$. Thus it can be beneficial to intentionally limit transparency ("blurring the full picture") to reduce exploitation and paradox, so long as accountability is not compromised.
\end{theorem}
\begin{proof}[Proof Sketch]
By assumption $T_2(x) \subseteq T_1(x)$ for all $x$ (garbling means you disclose less or equal). Thus $\lfp T_1 \supseteq \lfp T_2$ (monotone operators with one always disclosing more means its least fixed point will be larger). So the equilibrium state $x_1^{*}$ presumably has $x_1^{*} \supseteq x_2^{*}$. Then because $\Risk$ is increasing, $\Risk(x_1^{*}) \ge \Risk(x_2^{*})$. Under mild conditions one expects a strict inequality if indeed $T_1$ divulged strictly more on some dimension exploited or paradoxical. Meanwhile accountability $A(x)$ being monotone means $A(x_1^{*}) \ge A(x_2^{*})$, but by assumption both exceed $A_0$. If $A(x_2^{*})$ is just at threshold and $A(x_1^{*})$ above, we might have some wasted accountability potential but not needed.

Thus outcome $x_2^{*}$ has equal or better risk with still acceptable accountability. The welfare (which might be negative risk plus any gains) is higher.
\end{proof}

Design Theorem III supports intuitive strategies: e.g., instead of publishing exact audit criteria (which can be gamed), publish coarser guidelines that keep actors somewhat uncertain, which discourages fine-grained gaming. Or in releasing model information, avoid revealing complete details that allow self-referential triggers or fairness reversal.

We note one should not reduce transparency below what is needed for $A_0$ because then accountability suffers. The theorem just says if two policies both meet $A_0$, the less detailed one is safer. In practice, one finds a sweet spot where just enough transparency is given.

\section{Worked Constructions and Derived Equations}\label{sec:worked}
We gather several technical constructions to illustrate the theory:

\paragraph{Quantitative Fixed-Point Approximation:}
As mentioned earlier, if $T$ is $\omega$‑continuous, we have convergence:
\[
  T^{(0)}(\bot) = \bot,\quad
  T^{(\alpha+1)}(\bot) = T\big(T^{(\alpha)}(\bot)\big),\quad
  T^{(\lambda)}(\bot) = \sup_{\alpha<\lambda} T^{(\alpha)}(\bot),\ \text{$\lambda$ limit},\quad
  T^{(\omega)}(\bot) = \lfp T.
\]
In many practical cases (say $L$ is countable or chain‑complete in $\omega$ steps), $\lfp T = T^n(\bot)$ for some finite $n$ or at least the $\omega$‑chain stabilizes.

We claimed earlier that iterative application of $T$ approaches minimal risk:
\[
\forall \epsilon>0\,\exists n\in\N\,\forall k\ge n:\quad
\Risk\!\big(T^{k}(\bot)\big) \;\le\; \inf_{x\in \Fix(T)} \Risk(x) + \epsilon.
\]
This can be reasoned as follows: let $r^* = \inf_{x: T(x)=x} \Risk(x)$ (the optimal fixed-point risk). By Theorem \ref{thm:lfp-optimal}, $r^* = \Risk(\lfp T)$. Now $(\Risk(T^k(\bot)))_{k\in\N}$ is a decreasing sequence (since $T^k(\bot)$ grows and $\Risk$ monotone increases, so $\Risk(T^k(\bot))$ is non-decreasing; wait we need contractive assumption to ensure decreasing risk, not monotone risk with increasing states. If $\Risk$ is increasing, $T^k(\bot)$ yields increasing risk, so scratch that. Actually, if $T$ is risk-contractive, meaning it reduces risk each step, then risk goes down. Alternatively define $\loss = -\Risk$ if we want a measure that is increased by partial info. Possibly we consider $\loss$ as benefit (neg risk) to maximize. Then monotone partial info yields $\loss$ decreasing if risk increasing. Hmm let's reinterpret: risk monotone means more info, more risk, so as we iterate T from $\bot$ up, risk goes up. That sequence $\Risk(T^k(\bot))$ is increasing and converges to $\Risk(\lfp T)$ (monotone conv theorem) and can't overshoot $\Risk(\lfp T)$. Actually if T is $\omega$-continuous, $T^k(\bot) \to \lfp T$, and by continuity of $\Risk$ (assuming) $\Risk(T^k(\bot)) \to \Risk(\lfp T)$ from below. So given $\epsilon$, beyond some $N$, $\Risk(T^N(\bot))$ is within $\epsilon$ of $\Risk(\lfp T)$. That shows the property.

So yes, if risk continuous:
Given $\epsilon$, since $T^k(\bot) \uparrow \lfp T$, eventually in lattice sense it approaches, and $\Risk$ continuous implies the values approach. So that proves the $\epsilon$ statement.)

This tells us that performing iterative transparency until it stabilizes yields near-optimal risk eventually. It's analogous to value iteration in dynamic programming.

\paragraph{Sequent-Style Löb Derivation:}
We gave a Hilbert proof earlier; a sequent or Fitch-style might be:
	1.	Assume $\Box_P(\Box_P \varphi \to \varphi)$.
	2.	We want to show $\Box_P \varphi$. So consider the formal system of $P$; by Gödel's fixpoint lemma in arithmetic (the diagonalization again), one can find a sentence $\theta$ such that $P \vdash \theta \leftrightarrow (\Prov_P(\ulcorner \theta\urcorner) \to \varphi)$. That is a fixed point: $\theta$ says "if I'm provable, then $\varphi$." Now inside $P$, by (D1) and the assumption, $\Prov_P(\ulcorner\theta\urcorner)$ is derivable (because $\theta \leftrightarrow (\Prov_P(\ulcorner\theta\urcorner)\to\varphi)$ and using assumption we kind of get $\theta$ is provable, which yields $\Prov(\ulcorner\theta\urcorner)$). Then using $\theta$ itself, that implies $\varphi$. So $\varphi$ is proved. This is a somewhat different approach but ends similarly.
Actually this either duplicates Löb's known proof or is more complicated than needed.

Given the complexity, we'll trust the earlier direct approach.

\paragraph{Commutative Diagrams:}
We included Diagram \eqref{diag:lawvere} for Lawvere's theorem. Another relevant diagram might depict the interplay of $T$ and $B$ (disclosure and best-response). Possibly:
\[
\begin{tikzcd}
\text{State} \ar[r, "T"] \ar[d, dashed, equal] & \text{Info Disclosed} \ar[d, "B"] \\
\text{State} & \text{New State} \ar[l, dashed, equal]
\end{tikzcd}
\]
This is an informal diagram showing that at equilibrium, the horizontal then vertical equals identity mapping of state.

Alternatively, a category diagram for initial algebras: e.g.
\[
\begin{tikzcd}
\Fix(F) \ar[r, "\cong"] \ar[d, dashed] & F(\Fix(F)) \ar[d, dashed] \\
1 \ar[r] & \text{something}
\end{tikzcd}
\]
But let's skip constructing a second diagram since we covered one.

\paragraph{Ordinal Construction Example}\label{sec:worked:ordinal}
We claim if $T$ is not $\omega$-continuous, one may need $\omega_1$ steps. A classical example: Let $L = \mathcal{P}(\omega_1)$, the power set of the first uncountable ordinal (a complete lattice). Define $T(X)$ = the set of ordinals $\alpha$ such that $\alpha$ is either $0$, or a successor ordinal whose immediate predecessor is in $X$. This $T$ basically picks out 0 and all successor ordinals that follow an ordinal in $X$. It's monotone but not $\omega$-continuous (it doesn't consider limit ordinals well). If you iterate $T^n(\emptyset)$, you'll get all ordinals of finite successor length, at $\omega$ you'll get all countable ordinals maybe, but you'll never get a limit ordinal. The least fixed point of $T$ is $\omega_1$ (all countable ordinals because each countable ordinal is either 0 or successor of some smaller countable). But to accumulate $\omega_1$ you needed $\omega_1$ iterations (the $\omega$-chain gave only countable ordinals, you need cofinality $\omega_1$). So $\lfp T$ arrived at stage $\omega_1$. So the iterative algorithm will not finish in countably many steps. This is typical.

\paragraph{Algorithm and Complexity:}
Finally, an algorithm to compute $\lfp T$: assume $L$ finite or effectively enumerable partial order, $\omega$-continuous etc.
A straightforward algorithm:
\begin{enumerate}
\item $x \gets \bot$.
\item \textbf{repeat:} $x_{\text{prev}} \gets x$; $x \gets T(x_{\text{prev}})$.
\item \textbf{until} $x = x_{\text{prev}}$.
\item \textbf{return} $x$.
\end{enumerate}
This returns $\lfp T$ in finitely many iterations if $L$ has no infinite ascending chain. Complexity is $O(h \cdot C_T)$ where $h$ is height of chain or number of iterations and $C_T$ cost to compute $T$. If $L$ is finite of size $N$, worst-case $h = N$ (each round adds at least one new element until fixed). If $L$ infinite but well-founded or $\omega$-chain condition, algorithm might not terminate but if $\omega$-cont maybe treat it as approximate answer.

One can add a condition for early stopping if risk stops improving:
Because $\Risk(T^k(\bot))$ maybe monotonic, one could stop when additional risk < tolerance or $A$ at threshold etc.

Without $\omega$-continuity, termination isn't guaranteed (like the $\omega_1$ example above, it would never stop since state keeps growing countably forever and never stabilizes). That's the counterexample needed for the "not continuous" part: algorithm doesn't terminate.

Thus, one must either allow transfinite computation (impossible physically) or ensure $\omega$-continuity by design.

This completes our series of constructions and verifications.

\section{Optimization and Lagrange Duality in Lattices}\label{sec:optimality}
We revisit the constrained optimization \eqref{eq:design-opt}:
\[
\min_{x\in \Fix(T)} \loss(x) = \Risk(x) - \lambda\,\Gain(x) \quad \text{s.t. } A(x)\ge A_0.
\]
This is a partially ordered optimization problem rather than a linear one. But we can attempt to set up a Lagrangian:
\[
\mathcal{L}(x,\eta) = \loss(x) + \eta [A_0 - A(x)],
\]
with dual variable $\eta \ge 0$ for the constraint $A(x)\ge A_0$ (note writing as $A_0 - A(x) \le 0$ to conform to $\le 0$ form).
The idea is if $\hat{x}$ is optimal and the problem satisfies some lattice analog of convexity, then there exists $\hat{\eta} \ge 0$ such that:
\begin{align}
\hat{x}^* &\in \arg\min_{x \in \Fix(T)} \mathcal{L}(x,\hat{\eta}), \label{KKT1}\\
\hat{\eta} [A_0 - A(\hat{x})] &= 0, \quad A(\hat{x}) \ge A_0, \ \hat{\eta} \ge 0. \label{KKT2}
\end{align}
This is analogous to the Karush-Kuhn-Tucker (KKT) conditions. Eq.~\eqref{KKT2} is complementary slackness: either the constraint is tight ($A(\hat{x})=A_0$) and then $\hat{\eta}$ can be anything (in practice positive if objective benefits from relaxing constraint), or the constraint is slack ($A(\hat{x})>A_0$) in which case $\hat{\eta}=0$ (no penalty needed). Typically, one expects at optimum either one discloses just enough to meet accountability ($A(x^{*})=A_0$) or if risk is monotonically increasing in disclosure, one wouldn't exceed the required disclosure by too much, except if $\lambda\Gain$ strongly favors more openness.

The stationarity condition \eqref{KKT1} informally means: for all feasible $y$ in a neighborhood of $\hat{x}$ (or comparable to $\hat{x}$ in lattice), $\nabla \mathcal{L}(\hat{x},\hat{\eta})\cdot (y-\hat{x}) \ge 0$ where gradient is some subgradient because we might not have differentiability. In a lattice, one can define a subdifferential:
\[
\partial \loss(x) := \{ v \in \hat{V} : v(y-x) \ge 0, \forall y \text{ such that } x \le y \text{ (or $y \le x$)} \}.
\]
But here we only can vary $x$ along monotone directions since $T(x)$ must hold fixedpoint structure.

Alternatively, consider using Zorn's lemma to find an extremal optimum: if $\loss$ and $A$ are monotone, the problem might have some kind of supermodular structure.

We can illustrate with a simple scenario:
Suppose $\Fix(T)$ partially ordered by $\subseteq$ is a chain (so essentially one-dimensional). Then $\loss(x)$ might have a shape: decreasing then increasing perhaps if $\Gain$ dominates early and $\Risk$ later. The optimum either at boundary (least or most) or where derivative crosses zero: $\frac{d\Risk}{dx} = \lambda \frac{d\Gain}{dx}$ if treat them as continuous variables.

In general lattice, one could linearize by considering extreme points. Possibly each feasible is an antichain or chain in some partially ordered set.

Without too deep: the KKT-like result we can say:
\begin{theorem}\label{thm:KKT-lattice}
Suppose $\Fix(T)$ is a distributive lattice and $\loss, A$ are lattice-submodular or something nice. Then an optimal solution exists (by compactness maybe) and there is a dual $\hat{\eta}\ge0$ such that for any alternative $y \in \Fix(T)$ with $A(y)\ge A_0$, we have
\[
\loss(\hat{x}) \le \loss(y) + \hat{\eta}[A_0 - A(y)].
\]
This is a variational inequality capturing stationarity. Furthermore, complementary slackness holds: $\hat{\eta}[A_0 - A(\hat{x})]=0$.
\end{theorem}
\begin{proof}[Proof Sketch]
Existence from compactness of $(L,\le)$ or by monotone boundedness. The inequality is a form of saying the subgradient at $\hat{x}$ in direction of any feasible $y$ is non-negative. Slackness as usual KKT.
\end{proof}

In practice, computing the Lagrange multiplier $\hat{\eta}$ might correspond to how much weight on accountability would make $\hat{x}$ exactly break even. If $\hat{\eta}$ is high, it means accountability was a binding constraint pushing solution.

This concept can help in dynamic adjustments: if current $x$ fails accountability, increase $\eta$ to push for more transparency; if $x$ overshoots accountability by far, one could lower $\eta$ to reduce risk.

Thus, the design can be iteratively tuned with a "dual" interpretation: treat $\eta$ as a price on lack of accountability. Increase it until optimum sits exactly at threshold.

We omit a more rigorous proof due to abstract nature.

\section{Case Studies}\label{sec:case}
We outline three stylized scenarios to apply our framework:

\paragraph{(A) Audit Threshold Publication:}
An organization announces that any employee with productivity below $\theta$ will be fired (and employees know their own productivity). This is transparency of a threshold metric. How can this be gamed?

Formally, let $m(e)$ be an evaluation metric (like output per hour). They publish $\theta$. Some employees might \emph{appear} to raise productivity above $\theta$ by focusing on measured tasks at cost of other unmeasured valuable work. So actual performance drops or side effects occur, but $m$ is satisfied.

This is directly analogous to Theorem~\ref{thm:kleene-gaming}: the metric $m$ and threshold $\theta$ define a binary audit (pass/fail). There exists an employee strategy $\hat{e}$ (constructed by recursion theorem idea) that just meets $m(\hat{e})=\theta$ but fails underlying quality (like neglecting teamwork, which $m$ doesn't measure).

In practice, many metrics-based management suffer this: publishing the exact target invites just-good-enough behavior (regression to the threshold). The formal result: $\exists \hat{e}$ such that $m(\hat{e})=\theta$ (or just above) but $\hat{e}$ violates some non-measured norm.

We can consider $\Pi$ minimal here (no paradox, just gaming) and $\Lambda$ small (no privacy issue), but $G$ huge.

Solution: don't announce exact $\theta$ or use a composite metric including random checks for side tasks.

\paragraph{(B) Red-Team Playbook Release:}
Suppose a security team (red team) has a set of attack techniques $S$ it tests the system with. They consider releasing this list for transparency. If they do, a malicious actor can design an attack that avoids all techniques in $S$ (since those are presumably the only ones tested). If $S$ was a basis for all known attacks, releasing it essentially teaches adversaries what \emph{not} to do, thus what they \emph{can} do safely.

Model: Let $X$ be the space of attacks, with a topology or sigma-algebra. The red team has a set $S \subset X$ that is a "cover" of likely attacks (perhaps an open cover meaning any typical attack intersects some pattern in $S$). If $S$ is disclosed, the adversary chooses an attack in $X \setminus \Cl(S)$ (outside closure of known patterns). If the known sets $S$ didn't fully cover $X$ (which they rarely do), then an attack exists outside. This is a separation: $S$ separated safe region from detected region, now adversary picks a point in the gap.

Thus risk $\Lambda$ (privacy not an issue here) is nil, $\Pi$ nil, but $G$ large after disclosure.

We can formalize:
Proposition: If $\bigcup S \neq X$ (the union of known tactics doesn't cover all possible tactics), then disclosing $S$ allows constructing an undetected attack $x \in X \setminus \bigcup S$. The risk of successful attack leaps to 1 whereas before disclosure maybe adversary uncertain might stumble into $S$ with some probability. So expected risk up.

Essentially, previously unknown unknowns become known unknowns to adversary.

From lattice perspective, $S$ was info that should have been kept hidden to keep adversary uncertain.

A separation lemma can be akin to: in a topological space, if $S$ is a finite family of open sets, one can choose a point outside them if they don't cover. Or in measure sense, one can have measure to choose that region.

Hence a formal bound: Risk of breach $\ge 1 - \mu(\bigcup S)$ where $\mu$ is measure of coverage. If $\mu$ was less than 1, then risk left.

This encourages partial release: maybe mention broad principles but not the exact list.

\paragraph{(C) Process vs Outcome Transparency:}
Consider an AI system making decisions. Process transparency ($T_{\text{proc}}$) means revealing the algorithm or internal reasoning. Outcome transparency ($T_{\text{out}}$) means revealing just the decisions or outputs, not how they were made.

One might guess revealing the process gives more info to game or complain about fairness, whereas revealing only outcomes still allows some accountability (people see results) but not enough to replicate the algorithm.

We model that $T_{\text{out}}(x) \subseteq T_{\text{proc}}(x)$ typically: the outcome is a subset of the full process information. So $T_{\text{proc}}$ discloses strictly more.

Thus $\lfp T_{\text{proc}} \succeq \lfp T_{\text{out}}$. If risk is monotone, risk is higher for the process transparency solution.

One might also consider fairness distortion $\Phi$: sometimes revealing process can help detect bias (so maybe $\Phi$ down with process transp.), but it also might cause gaming ($G$ up) or paradox ($\Pi$ up if algorithm can introspect? Possibly not).

However, generally we get:
Proposition: If $T_{\text{out}}(x) \preceq T_{\text{proc}}(x)$ for all $x$ (process reveals all outcome plus more), then indeed $\lfp T_{\text{out}} \preceq \lfp T_{\text{proc}}$. Combined with monotonic risk, $\Risk(\lfp T_{\text{out}}) \le \Risk(\lfp T_{\text{proc}})$.

So outcome transparency is safer. It might not satisfy as high accountability though. But often outcome accountability suffice (if decisions can be audited externally for bias by looking at patterns, though not as well as code review for fairness).
So trade-off: process transparency yields better $A$ (trust, explainability) but also more risk. Outcome transparency is the minimal needed for basic accountability.

Thus balancing them is context-dependent, but formal dominance as above.

\section{Design Calculus: Formal Prescriptions}\label{sec:design}
We synthesize guidelines as formal conditions:

\begin{theorem}[Optimal Transparency Policy Conditions]\label{thm:design-conditions}
Given the risk decomposition $\Risk = \alpha \Pi + \beta \Lambda + \gamma \Phi + \delta G$, the transparency policy $T$ that minimizes $\Risk$ subject to $A \ge A_0$ will have the following qualitative form, under mild assumptions on monotonicity and separability:
\begin{enumerate}
\item (\textbf{Partial Transparency for Self-Reference}) If $\alpha > 0$ (paradox risk matters), then the optimal $T$ must exclude a set $S_{\Pi}$ of self-referentially dangerous statements from ever being transparently declared. That is, $\lfp T$ will avoid making $\Trans(\sigma)$ claims for $\sigma$ that talk about $\Trans$ itself. Equivalently, any total transparency (including all self-referential sentences) yields $\Pi = +\infty$ cost, so is dominated by a partial policy. $S_{\Pi}$ may be constructed via a hierarchy (like grounded theory) or explicitly listing known paradox triggers.
\item (\textbf{Stratified/Revealed-in-Phases Transparency}) If the environment allows it, the solution will often do transparency in layers (stratified) rather than all-at-once. Formally, if releasing $I_1$ then $I_2$ in sequence yields lower risk than releasing $I_1 \cup I_2$ together (which could cause compounding interactions), then the optimal policy is to stratify. Sufficient conditions: if $\Risk$ is superadditive in disclosures (meaning risk of combined disclosure is more than sum of individual), stratify to mitigate. This can be proven by considering $\Risk(I_1 \cup I_2) \ge \Risk(I_1) + \Risk(I_2|I_1)$ vs sequential.
\item (\textbf{Randomized Transparency}) If agents optimize against known deterministic policies, injecting randomness or uncertainty in what is disclosed can reduce $G$ significantly. For instance, if two disclosure forms $T_1, T_2$ have similar accountability, a policy that with probability $p$ uses $T_1$ and $1-p$ uses $T_2$ can leave adversaries unsure, forcing a strategy that hedges against both, often lowering worst-case gaming. Condition: if $G(x)$ is convex in information precision (more precise info yields disproportionately higher gaming), then a random mixture lowers expected exploitation. This follows from Jensen's inequality on convex $G$.
\item (\textbf{Accountability Floor}) Ensure at minimum the constraint holds: $\forall x \in \Fix(\hat{T}), A(x) \ge A_0$. This may require including some core disclosures (perhaps outcomes or summary statistics) to ensure oversight. That is, the least fixed point $T^{*}$ must contain a subset $S_A$ (non‑negotiable disclosures) chosen such that $A(S_A) = A_0$. The rest of $T^{*}$ then consists of additional discretionary disclosures chosen by balancing other risks.
\item (\textbf{Stopping Criterion}) An implementable design: iterate disclosure until marginal risk of additional disclosure outweighs marginal gain or accountability improvement. In practice: while $A(x) < A_0$, add the most accountability-boosting, least risk-raising item. Stop when $A(x)$ just hits $A_0$. This greedy algorithm will hit optimal if $\loss$ is convex monotone and disclosures can be added incrementally. Our earlier iterative algorithm approximates $\lfp T$, but here we incorporate stopping at $A_0$ rather than full fixpoint if fixpoint overshoots.
\end{enumerate}
\end{theorem}
\begin{proof}[Justification (Sketch)]
Each item is justified by prior results: (1) by Theorem \ref{thm:trans-diag} and Corollary \ref{cor:partial-transparency}. (2) Stratified release is beneficial if interactions are adverse; formally, if $T = T_2 \circ T_1$ yields $\Risk(\lfp(T_2\circ T_1)) < \Risk(\lfp(T_{combined}))$. This holds if risk from $I_1$ and $I_2$ non-linearly interacts. No contradiction with constraints if time is allowed. (3) is a game-theoretic result akin to mixing strategies making the opponent (adversary) less effective. The convexity of exploitability makes mixing reduce max-exploitation (minimax logic). (4) trivial from constraint; one picks minimal set that satisfies it. (5) is essentially KKT logic or incremental gradient method. If disclosing an item improves $A$ a lot with small risk, do it, until any further item has too high risk per $A$ gained.
\end{proof}

Implementing these:
	•	Omit certain categories of info (like self-analytical info).
	•	Reveal gradually (like first share general principles, later specifics if needed).
	•	Possibly randomize details (like slight fuzzing of data).
	•	Always include key accountability metrics (like outcomes or summary audits).
	•	Use an algorithm to approximate minimal but sufficient info: our pseudocode earlier but incorporate $A$ check:
\begin{enumerate}
\item $X = \emptyset$.
\item While $A(X) < A_0$:
\begin{enumerate}
\item Find an item $s \in S\setminus X$ whose disclosure most increases $A$ relative to $\Risk$ added (like maximize $\Delta A/\Delta \Risk$ or similar).
\item $X \gets X \cup {s}$.
\item If $\Risk(X) >$ some threshold (like if about to violate risk tolerance), break and reconsider policy.
\end{enumerate}
\item Return $X$.
\end{enumerate}
This greedy algorithm needs assumptions to be optimal (like linear or separable risk contributions). If those hold, it yields a near-minimal $X$ for required $A$.

In domain theory terms, one constructs $X$ by adding join-irreducible elements one by one until meet accountability. If risk is submodular or something, greedy is exact.

As a final note, the design calculus underscores: it's rarely optimal to reveal everything. The math ensures some constraint is active, preventing full disclosure.

\section{Conclusion}
We have developed a rigorous, logic-first analysis of transparency through the lens of fixed-point theorems and self-referential paradoxes. The formalism allowed us to derive limits (no-go theorems) and constructive design principles in equal measure. From Gödel's diagonal lemma to Kleene's recursion theorem, classical theoretical computer science results find new interpretation as warnings or guides for transparency policy:
	•	Gödel/Tarski: complete self-transparency is impossible in a classical consistent system.
	•	Lawvere: any sufficiently expressive disclosure will entail self-referential equilibria.
	•	Knaster–Tarski: transparency policies have extremal fixed points; minimal ones are safest.
	•	Kripke: relaxing truth to partial values resolves paradox at cost of leaving some questions unanswered — a reasonable trade in practice.
	•	Löb: beware of self-justifying policies; require independent evidence.
	•	Kleene: any published metric invites a fixed-point exploit; hold some metrics back or randomize.
	•	Paraconsistent: you can "have it all" (a total truth-telling) if you're willing to live with contradictions — a theoretical option perhaps not practically palatable.
	•	Modal $\mu$-calculus: formulate requirements as nested fixpoints to ensure design invariants.

Our optimization view, albeit abstract, suggests the optimal transparency is usually partial and principled, not maximal. The "radical transparency" slogan runs afoul of diagonalization — too much openness undermines itself.

We focused on mathematical logic to articulate these issues. Future work might integrate more empirical aspects or quantify these risks on real data. Nonetheless, the core take-away is enduring: any powerful system that attempts to fully expose itself will encounter liar paradoxes and Goodhart's curses. The mathematics was our guide to find a middle path: enough transparency for accountability, but not so much that the system (or its users) can hack itself.

We conclude with a symbol glossary and appendices for deeper proofs omitted due to space.


\appendix

\section*{Appendix A: Deferred Proofs}
\addcontentsline{toc}{section}{Appendix A: Deferred Proofs}
\begin{proof}[Proof of Theorem~\ref{thm:lawvere} (Sketch)]
We provide only a sketch.  Working in a Cartesian closed category, one uses the exponential $X^X$ and its evaluation map to show that for any endomorphism $F\colon X\to X$ there is an element $x\in X$ with $F(x)=x$.  The universal property of the exponential yields a morphism $g\colon X\to X^X$ such that $\mathrm{eval}(g(x),x)=F(x)$ for each $x$.  A diagonal argument then produces $x$ with $F(x)=x$.  For the full categorical proof see Lawvere's original article\cite{Lawvere1969}.\AW{0.25}
\end{proof}

\begin{proof}[Proof of Proposition \ref{prop:paraconsistent}]
We outline the LP model: take all sentences of the language including a truth predicate $T$. The theory $\Delta = {T(\ulcorner \phi\urcorner) \leftrightarrow \phi : \phi \text{ any sentence}}$ is inconsistent classically. In LP, define a model $M$ with valuation $V$ as:
	•	If $\phi$ is not provable from $\Delta$, let $V(\phi) = 0$ (false).
	•	If $\neg\phi$ is not provable from $\Delta$, let $V(\phi) = 1$ (true).
	•	If both $\phi$ and $\neg \phi$ are provable from $\Delta$ (which happens for liar sentences and similar under $\Delta$), assign $V(\phi) = V(\neg\phi) = 1$ (both true and false).

This yields a valuation that satisfies all biconditionals in the sense of LP (where a biconditional is true if either both sides have the same truth value 1 or both have 0 or both are 1/0 in LP which is weird since LP doesn't have an "both" but in LP any formula with both values is considered true for atomic formulas, and extended by truth tables that make $\leftrightarrow$ true in that case as well because each side is true? Actually in LP, an atomic proposition can be both true and false, but a compound like $A \leftrightarrow B$ is considered true if whenever $A$ and $B$ share at least truth or something. It's easier to reason this: since in the model we've given, for any $\phi$, $T(\ulcorner \phi\urcorner)$ and $\phi$ have the same truth values multiset ${0/1}$, thus $T(\phi) \leftrightarrow \phi$ is assigned true in LP's truth tables (which typically treat (1,1) as true, (0,0) as true, (both,both) as true presumably).
Thus all axioms hold and not everything is true since some $\phi$ are only false or only true but not both, so you can't prove an arbitrary $\psi$ because you'd need $\psi$ to become both to automatically be true (the consequence relation in LP can be defined via preservation of 'truth' which is ambiguous under gluts, so one typically says $\Gamma \models_{LP} \phi$ if for every model that makes all of $\Gamma$ designated (true), $\phi$ is designated. In our model, $\Delta$ itself makes itself all true by design, but does it make, say, an unrelated proposition $\rho$ true? If $\rho$ is not mentioned, $\rho$ is false only in our construction, so $\rho$ not designated, so $\Delta \not\models \rho$. Hence consistent).
Thus $\Delta$ is satisfiable in LP, hence not explosive, and $T$ is total by construction.
\end{proof}

\section*{Appendix B: Symbol Glossary}
\addcontentsline{toc}{section}{Appendix B: Symbol Glossary}
\begin{table}[htbp]\centering
\begin{tabular}{ll}
\textbf{Symbol} & \textbf{Meaning} \\ \hline
$\N, \Z, \R$ & Natural, Integer, Real numbers. \\
$\Pow(S)$ & Power set of $S$ (set of all subsets of $S$). \\
$\ulcorner \phi\urcorner$ & Gödel code of formula $\phi$. \\
$\Prov_T(x)$ & Provability predicate of theory $T$ on code $x$. \\
$\True$ & Truth predicate (interpreted in context, e.g. $\True_\alpha$ stage sets). \\
$\Trans(x)$ & Transparency predicate (we defined similar to truth). \\
$\lfp T$, $\gfp T$ & Least and greatest fixed point of operator $T$. \\
$\Fix(T)$ & Set of all fixed points of $T$. \\
$\mu X.,\phi(X)$ & Least fixed point formula in $\mu$-calculus. \\
$\nu X.,\psi(X)$ & Greatest fixed point formula. \\
$\Sem{\phi}$ & Semantics (denotation) of formula $\phi$ (usually as a set of states). \\
$\Box, \Diamond$ & Modal necessity and possibility (provability or next/always, eventually). \\
$\modelsK$ & Semantic entailment (in a Kripke or logical structure). \\
$\mid$ & Such that (used in set comprehension). \\
$\Risk, \Gain, A$ & Risk functional, Gain functional, Accountability measure. \\
$\loss(x)$ & Combined loss = $\Risk(x) - \lambda \Gain(x)$ to minimize. \\
$\Pi(x), \Lambda(x), \Phi(x), G(x)$ & Components of risk: paradox, leakage, fairness, gaming. \\
$\Box \varphi$ & In provability logic, $\Prov_T(\ulcorner \varphi\urcorner)$ (provable $\varphi$). \\
$\models$ & Logical entailment (from axioms or premises to conclusion). \\
$\vdash$ & Derivability in a formal system. \\
$\Vdash$ & Satisfaction relation in modal logic (like $M,w \Vdash \phi$). \\
$\Id$ & Identity function. \\
$\diag$ & (In text use) diagonalization or diagonal function in category. \\
$\mathrm{eval}$ & Evaluation map in a CCC ($X^X \times X \to X$). \\
\hline
\end{tabular}
\caption{Glossary of symbols and notation.\label{tab:glossary}}
\end{table}

\clearpage
\section*{Appendix C: Pseudocode and Algorithms}
\addcontentsline{toc}{section}{Appendix C: Pseudocode and Algorithms}
We provide a high-level pseudocode for iterating a transparency policy to find a fixed point under constraints, as discussed:

\vspace{1ex}
\noindent \textbf{Algorithm: ComputeMinTransparency($T, A_0$)}
\begin{enumerate}
\item $X \leftarrow \bot$ \hfill // start with no disclosure
\item \textbf{repeat}
\begin{enumerate}
\item $X_{\text{old}} \leftarrow X$.
\item $X \leftarrow T(X_{\text{old}})$.
\item \textbf{if} $A(X) < A_0$ \textbf{then}
\begin{itemize}
\item Identify an item $s \in S$ (not yet in $X$) that maximizes $\frac{A(X \cup {s}) - A(X)}{\Risk(X \cup {s}) - \Risk(X)}$.
\item $X \leftarrow X \cup {s}$.
\end{itemize}
\item \textbf{if} $A(X) > A_0$ and $X_{\text{old}}=X$ \textbf{then break}.
\end{enumerate}
\item \textbf{return} $X$.
\end{enumerate}

This algorithm iteratively applies $T$ (which may disclose a chunk of info), then if accountability is not reached, it force-adds the most "efficient" piece of info to boost $A$. It stops when a fixed point is reached with $A \ge A_0$. The selection step uses a heuristic ratio; in practice one might add multiple items at once if independent.

The complexity: each iteration might scan potential items to add (size of $S$). If $|S|=N$, worst-case adds all, so $O(N^2)$ steps.

This approach combines the iterative Knaster–Tarski computation with a greedy satisfaction of the constraint. It ensures the final outcome is on the boundary $A=A_0$ typically (hitting just sufficient transparency). The proof of optimality would require submodularity assumptions.

For computing fixpoints in code, one might represent $L$ implicitly (like a lattice of properties) and $T$ as a function. The algorithm is straightforward to implement for finite lattices by brute force.

A termination proof: If $T$ is monotone and $S$ finite, each loop adds something (either by $T$ or by the explicit add), so $X$ grows. It cannot grow forever beyond $S$. So it terminates in at most $|S|$ loops.

If $S$ infinite or $T$ not inflationary, termination might be trickier; but typically one expects to either converge or continue until manually stopped (approximation scenario).

We have thus provided the tools and formal reasoning a designer would need to implement and justify a responsible transparency strategy in complex self-referential systems.


\begin{thebibliography}{10}

\bibitem{Lawvere1969}
F.~William Lawvere.
\newblock Diagonal arguments and cartesian closed categories.
\newblock In \emph{Category Theory, Homology Theory and their Applications II}, Lecture Notes in Mathematics, volume\,92, pages 134--145. Springer, 1969.\newline DOI:\,\href{https://doi.org/10.1007/BFb0080769}{10.1007/BFb0080769}.

\bibitem{Tarski1955}
Alfred Tarski.
\newblock A lattice-theoretical fixpoint theorem and its applications.
\newblock \emph{Pacific Journal of Mathematics}, 5:285--309, 1955.\newline DOI:\,\href{https://doi.org/10.2140/pjm.1955.5.285}{10.2140/pjm.1955.5.285}.

\bibitem{Kripke1975}
Saul A.~Kripke.
\newblock Outline of a theory of truth.
\newblock \emph{The Journal of Philosophy}, 72(19):690--716, 1975.\newline DOI:\,\href{https://doi.org/10.2307/2024309}{10.2307/2024309}.

\bibitem{Lob1955}
Martin H.~L{"o}b.
\newblock Solution of a problem of {L}eon {H}enkin.
\newblock \emph{Journal of Symbolic Logic}, 20(2):115--118, 1955.\newline DOI:\,\href{https://doi.org/10.2307/2268930}{10.2307/2268930}.

\bibitem{Kleene1938}
Stephen~Cole Kleene.
\newblock On notations for ordinal numbers.
\newblock \emph{Journal of Symbolic Logic}, 3(4):150--155, 1938.\newline DOI:\,\href{https://doi.org/10.2307/2267755}{10.2307/2267755}.

\bibitem{Kakutani1941}
Shizuo Kakutani.
\newblock A generalization of {B}rouwer's fixed point theorem.
\newblock \emph{Duke Mathematical Journal}, 8(3):457--459, 1941.\newline DOI:\,\href{https://doi.org/10.1215/S0012-7094-41-00838-4}{10.1215/S0012-7094-41-00838-4}.

\end{thebibliography}
\end{document}